\documentclass[a4paper,11pt]{article}

\usepackage{amsmath}
\usepackage{amsthm}
\usepackage{amssymb}
\usepackage{amsfonts}
\usepackage{mathrsfs}
\usepackage{here}
\usepackage{graphicx}
\usepackage{xcolor}
\usepackage{latexsym}
\usepackage{enumerate}
\usepackage[T1]{fontenc} 
\usepackage{authblk} 

\definecolor{mypink}{rgb}{0.9, 0.0, 0.4}
\definecolor{mygreen}{rgb}{0.0,0.5,0.0}
\definecolor{mypurple}{rgb}{0.6,0.0,0.6}
\definecolor{myfuchsia}{rgb}{0.9,0.0,0.9}
\definecolor{mybrown}{rgb}{0.7,0.3,0.3}
\definecolor{mygray}{rgb}{0.6,0.6,0.6}
\definecolor{myltgray}{rgb}{0.85,0.85,0.85}

\newcommand{\upd}[1]{#1} 

\theoremstyle{definition}
\newtheorem{definition}{Definition}[section]
\newtheorem{theorem}[definition]{Theorem}
\newtheorem{proposition}[definition]{Proposition}
\newtheorem{corollary}[definition]{Corollary}
\newtheorem{lemma}[definition]{Lemma}

\newtheorem{remark}[definition]{Remark}

\title{A characterization of terminal planar networks by forbidden \upd{structures}
\thanks{\upd{This work is based on the Master's and Bachelor's theses of Haruki Miyaji and Yuki Noguchi, respectively, supervised by Momoko Hayamizu.}}}

\author[1]{Haruki Miyaji}
\author[2]{Yuki Noguchi}
\author[1]{Hexuan Liu}
\author[1]{Takatora Suzuki}
\author[1]{Keita Watanabe}
\author[3]{Taoyang Wu}
\author[4]{Momoko Hayamizu
\thanks{Corresponding author: \texttt{hayamizu@waseda.jp}}}

\affil[1]{Department of Pure and Applied Mathematics, Graduate School of Fundamental Science and Engineering, Waseda University, Tokyo, Japan}
\affil[2]{Department of Applied Mathematics, School of Fundamental Science and Engineering, Waseda University, Tokyo, Japan}
\affil[3]{School of Computing Science, University of East Anglia, Norwich, UK}
\affil[4]{Department of Applied Mathematics, Faculty of Science and Engineering, Waseda University, Tokyo, Japan}

\date{\vspace{-3em}} 

\begin{document}
\maketitle 
\begin{abstract}
The class of terminal planar networks was recently introduced from a biological perspective \upd{in relation to the visualization of phylogenetic networks}, and its connection to upward planar networks has been established. 
We provide a Kuratowski-type theorem that characterizes terminal planar networks by a finite set of forbidden structures, defined via six families of 0/1-labeled graphs. Another characterization \upd{based on planarity of supergraphs} yields linear-time algorithms for testing terminal planarity and for computing such planar drawings. 
We describe an application that is potentially relevant in broader, \upd{non-phylogenetic} settings. We also discuss a connection of our main result to an open problem on the forbidden structures of single-source upward planar networks. 
\end{abstract}

\section{Introduction}
\upd{
Finding the set of forbidden structures that characterizes a graph class is a fundamental topic in graph theory. The well-known Kuratowski's theorem \cite{Kuratowski1930} states that a graph is planar if and only if it contains no subgraph homeomorphic to $K_5$ or $K_{3,3}$. This type of characterization is known for different graph classes (e.g., \cite{AIHPB_1967__3_4_433_0,Wagner1937,d60053e7-18fc-319e-8b97-0078fea6e467}; see Chapter 7 of \cite{doi:10.1137/1.9780898719796} for details). However, such forbidden structures are unknown for several important graph classes, including upward planar networks and terminal planar networks.
}

\upd{
Upward planar digraphs are a classical and well-studied class in graph drawing (e.g., \cite{HL1996upss,bbmt-oupsst-1998,DIBATTISTA1988175, garg1995upward}). A digraph is upward planar if it admits a planar drawing in which all edges are monotone upward with respect to the vertical direction. For embedded digraphs with a single source, upward planarity has been characterized by Thomassen~\cite{Thomassen1989PlanarAcyclicOriented} in terms of certain forbidden patterns, and this characterization is used in the algorithms in \cite{HL1996upss} (see Theorem 6.13 in Section 6.7 of \cite{BETT1998GraphDrawing} for details). Despite this understanding of the embedded case, neither forbidden subgraphs nor forbidden minors are currently known for upward planar digraphs.
}

\upd{
Terminal planar networks were more recently introduced by Moulton--Wu~\cite{WuMoulton9804871} as a subclass of directed phylogenetic networks motivated by the visualization of evolutionary relationships among organisms. 
A directed phylogenetic network is an acyclic digraph with a unique root and one or more leaves, and it is terminal planar if it admits a planar embedding such that all terminals (the root and leaves) lie on the outer face. In \cite{WuMoulton9804871}, it was shown that a directed phylogenetic network is terminal planar if and only if a certain single-source supergraph of it---which we call the $t$-completion in this paper---is upward planar, and linear-time algorithms for terminal planarity testing and drawing were described. However, the characterization based on upward planarity of $t$-completion does not easily yield a Kuratowski-type characterization of terminal planar networks, since the forbidden structures of upward planar networks are not fully understood, as mentioned above. Another limitation of this approach is that it does not apply to undirected graphs, because of the inherently directed nature of upward planarity.
}

\upd{
In this paper, we provide a characterization of terminal planar networks in terms of a finite set of forbidden structures, called $\mathcal{H}_1, \dots, \mathcal{H}_6$ structures (Theorem \ref{thm:main}), which is valid for undirected graphs as well (Corollary  \ref{cor:main_underlying}).  This is indeed a Kuratowski-type theorem because $\mathcal{H}_1$ and $\mathcal{H}_4$ structures are nothing but subdivisions of $K_{3,3}$ and $K_5$, respectively. When we only consider binary phylogenetic networks,  $\mathcal{H}_1, \mathcal{H}_2,$ and $\mathcal{H}_3$ structures are the only possible obstructions (Corollaries \ref{cor:main_binary} and \ref{cor:main_binary_underlying}).  
}

\upd{
To prove the main theorem, we first show the equivalence among the terminal planarity of directed phylogenetic networks and the planarity of their supergraphs called the $st$-completion and the terminal cut-completion (Theorem~\ref{lem:N_completion}). 
This characterization plays an important role in proving the main theorem and also furnishes new linear-time algorithms for testing terminal planarity and for drawing such networks since it links terminal planarity with planarity, not with upward planarity. 
}

\upd{
The remainder of the paper is organized as follows. Section~\ref{sec:preliminaries} introduces the basic definitions and notation of graph theory,  and Section~\ref{sec:known_results} reviews the concepts and known results on graph planarity, including upward planarity and terminal planarity of phylogenetic networks. In Section~\ref{sec:characterization.transformation}, in addition to proving Theorem~\ref{lem:N_completion}, we also briefly mention the relationship between terminal planar directed phylogenetic networks and bar-visibility digraphs. 
In Section~\ref{sec:characterization.forbidden}, after setting up the necessary terminology, we prove Theorem~\ref{thm:main} and also describe linear-time algorithms for terminal planarity testing and for drawing such networks. 
In Section~\ref{sec:application}, we demonstrate an application of the main result in more general, non-phylogenetic settings by providing a characterization of planar graphs with an embedding such that specified vertices lie on the outer face (Theorem~\ref{thm:general}).
Finally, in Section~\ref{sec:conclusion_open_problems}, we give a conclusion and briefly discuss how our results may be relevant to the open problem of determining the forbidden structures for upward planar networks.
}

\section{\upd{Basic definitions and notation}\label{sec:preliminaries}}
\subsection{Undirected graphs}
An \emph{undirected graph} is defined by an ordered pair $(V, E)$ consisting of a set $V$ of vertices and a set $E$ of undirected edges. Given an undirected graph $G$, $V(G)$ and $E(G)$ denote the vertex set and edge set of $G$, respectively. A graph $G$ is \emph{finite} if both $V(G)$ and $E(G)$ are finite sets. Given an edge $\{u,v\}$ of an undirected graph, the vertices $u$ and $v$ are called the \emph{endpoints} of the edge.  An edge $\{u,v\}$ with $u = v$ is called a \emph{loop}.
Two or more edges between the same pair of vertices are called \emph{multiple edges}. An undirected graph is \emph{simple} if it contains neither loops nor multiple edges. In this paper, we only consider finite simple undirected graphs.

For a vertex $v$ and an edge $e$ of a graph, if $v$ is an endpoint of $e$, we say that $e$ is \emph{incident} to $v$. For two distinct vertices $u$ and $v$ of a graph $G$, if $\{u,v\}\in E(G)$, then $u$ and $v$ are said to be \emph{adjacent}. For any vertex $v$ of a graph $G$, $\mathrm{Adj}_G(v)$ denotes the set of all vertices adjacent to $v$. The value $|\mathrm{Adj}_G(v)|$ is called the \emph{degree of $v$ in $G$} and is denoted by $\deg_G(v)$. A vertex of $G$ with degree 1 is called a \emph{leaf} or \emph{terminal vertex} of $G$, and an edge incident to a leaf is called a \emph{pendant edge} of $G$.

A graph $P$ with $V(P)=\{v_1,v_2, \dots, v_k\}$ and $E(P)=\{\{v_1,v_2\},\{v_2,v_3\}, \dots, \{v_{k-1},v_k\}\}$ is called a \emph{path}. In this case, the vertices $v_1$ and $v_k$ are called the \emph{endpoints} of $P$, and the other vertices are called the \emph{internal vertices} of $P$. The value $k-1$ is the \emph{length} of $P$ (where $k\geq 2$). A path can be represented as an alternating sequence of vertices and edges $v_1,\{v_1,v_2\}, v_2, \dots, v_{k-1}, \{v_{k-1},v_k\}, v_k$, but when the context is clear, we often write  $v_1-v_2-\dots-v_k$. A path with the endpoints $u$ and $v$ is called a \emph{$u$-$v$ path} and may be denoted by $[u,v]$. A graph $C$ with $V(C)=\{v_1,v_2, \dots,v_k\}$ and $E(C)=\{\{v_1,v_2\},\{v_2,v_3\}, \dots,\{v_{k-1},v_k\}, \{v_k,v_1\}\}$ such that all vertices are distinct is called a \emph{cycle} ($k \geq 3$).

Two graphs $G$ and $H$ are \emph{isomorphic}, denoted by $G \simeq H$, if there exists a bijection $\phi\colon V(G) \to V(H)$ such that for any $u,v \in V(G)$, we have $\{u,v\} \in E(G)$ if and only if $\{\phi(u),\phi(v)\} \in E(H)$. 
If $V(G) \subseteq V(H)$ and $E(G) \subseteq E(H)$, then $G$ is a \emph{subgraph} of $H$, and we say $H$ \emph{contains} $G$, or $H$ is a \emph{supergraph} of $G$. This relationship is denoted by $G \subseteq H$. In particular, if $G\subseteq H$ and $V(G)=V(H)$, then $G$ is a \emph{spanning} subgraph of $H$. If $G\subseteq H$ and $G\not \simeq H$, then we write $G \subsetneq H$.  In this case, we call $G$ a \emph{proper} subgraph of $H$ and  $H$  a \emph{proper} supergraph of $G$. The \emph{union} of two graphs $G$ and $H$, denoted by $G \cup H$, is the graph with vertex set $V(G) \cup V(H)$ and edge set $E(G) \cup E(H)$. Similarly, the \emph{intersection} $G \cap H$ is the graph with vertex set $V(G) \cap V(H)$ and edge set $E(G) \cap E(H)$.

A graph $G$ is \emph{connected} if there exists a $u$-$v$ path for any $u,v \in V(G)$. We  write $[u,v]_G$ to mean a $u$-$v$ path in a graph $G$. For two vertices $u$ and $v$ of a graph $G$, the \emph{distance} between $u$ and $v$ in $G$ is defined by the length of a shortest $u$-$v$ path in $G$. A \emph{connected component} of $G$ is a connected subgraph of $G$ that is not a proper subgraph of any other connected subgraph of $G$. Given a graph $G$, the graph obtained by removing an edge $e$ from $G$ is denoted by $G-e$, and this operation is called the \emph{deletion} of $e$. For any $E^\prime\subset E(G)$, the graph $G-E^\prime$ denotes the graph that is obtained by deleting all edges in $E^\prime$ from $G$. 
 Similarly, the graph obtained by removing a vertex $v$ together with all edges incident to $v$ from $G$ is denoted by $G - v$, and this operation is called the \emph{deletion} of $v$. For any $V^\prime\subseteq V(G)$, the graph $G-V^\prime$ denotes the graph that is obtained by deleting all vertices in $V^\prime$.
 
  For a graph $G$, $e\in E(G)$ is called a \emph{cut edge} of $G$ if $G-e$ is disconnected. Also, $v\in V(G)$ is called a \emph{cut vertex} of $G$ if $G-v$ is disconnected. An edge $e$ of an undirected graph $G$ is a cut edge of $G$ if and only if $e$ is not contained in any cycle of $G$ (e.g., \cite[Theorem 2.3]{bondymurty1976}). A graph $G$ is \emph{biconnected} if it contains no cut vertices. This implies that a biconnected graph does not contain any cut edges. A biconnected subgraph of $G$ that is not a proper subgraph of any other biconnected subgraph is called a \emph{block} of $G$.

For an edge $\{u,v\}$ of a graph $G$, the \emph{subdivision} of $\{u,v\}$ is the operation of replacing $\{u,v\}$ with a $u$-$v$ path of length at least $1$. A graph obtained by a finite sequence of edge subdivisions is called a \emph{subdivision} of $G$. For a vertex $v$ of a graph $G$ with $\deg_G(v)=2$, and its incident edges $e_1 = \{u_1, v\}$ and $e_2 = \{u_2, v\}$, the \emph{smoothing} of $v$ is the operation whereby $v$, $e_1$, and $e_2$ are deleted, and a new edge $\{u_1, u_2\}$ is added instead. 
Note that since we only consider simple graphs in this paper, if smoothing results in multiple edges, they are always replaced by a single edge. The simple graph obtained by smoothing all possible vertices of a graph $G$ is called the \emph{smoothed graph of $G$}. Two graphs $G$ and $H$ are \emph{homeomorphic}, denoted by $G \approx H$, if there exist subdivisions $G^\prime$ of $G$ and $H^\prime$ of $H$ such that $G^\prime \simeq H^\prime$ holds.

For an edge $e = \{u, v\}$ in a graph $G$, the \emph{contraction} of $e$ is the operation that merges $u$ and $v$ into a single new vertex, replaces $e$ by this vertex, and connects it to all neighbors of $u$ or $v$ while all other vertices and edges are kept unchanged. The resulting graph is denoted by $G / e$. As in smoothing, if contraction creates multiple edges, they are replaced by a single edge. More generally, for a set of edges $E^\prime \subseteq E(G)$, the graph obtained by contracting all edges in $E^\prime$ is denoted by $G / E^\prime$.
A graph $G^\prime$ is called a \emph{minor} of a graph $G$ if $G^\prime$ can be obtained from $G$ by a sequence of vertex deletions, edge deletions, and edge contractions. In this case, we say that $G$ \emph{contains} $G^\prime$ as a minor.

A graph with $n$ vertices in which every pair of distinct vertices is adjacent is called a \emph{complete graph} and is  denoted by $K_n$. A graph $G$ is \emph{bipartite} if $V(G)$ can be partitioned into two non-empty subsets $X$ and $Y$ such that every edge of $G$ has one endpoint in $X$ and the other in $Y$. The sets $X$ and $Y$ are called the \emph{partite sets} of $G$. A bipartite graph $G$ with partite sets $X$ and $Y$ with $|X|=p$ and $|Y|=q$ is called a \emph{complete} bipartite graph, denoted by $K_{p,q}$, if every vertex in $X$ is adjacent to every vertex in $Y$.

\subsection{Directed graphs}
A \emph{directed graph}, or \emph{digraph}, is an ordered pair $D=(V,A)$, where $V$ is a set of vertices and $A$ is a set of directed edges called \emph{arcs}. For a digraph $D$, its vertex set and arc set are denoted by $V(D)$ and by $A(D)$, respectively. A digraph $D$ is \emph{finite} if both $V(D)$ and $A(D)$ are finite sets. An arc is an ordered pair of vertices $(u,v)$; $u$ is called the \emph{tail} and $v$ the \emph{head} of the arc. A \emph{simple} digraph has no loops or parallel arcs. In this paper, all digraphs are assumed to be finite and simple.

Two digraphs $D$ and $H$ are \emph{isomorphic}, denoted by $D \simeq H$, if there exists a bijection $\phi\colon V(D) \to V(H)$ such that $(u,v) \in A(D)$ if and only if $(\phi(u),\phi(v)) \in A(H)$ for all $u,v \in V(D)$.  
The concepts of subgraphs, proper subgraphs, spanning subgraphs, supergraphs, and proper supergraphs are defined similarly to the undirected case.

The \emph{underlying graph} of a digraph $D$, denoted $D_u$, is the (simple) undirected graph obtained from $D$ by replacing each arc $(u,v)\in A(D)$ with an undirected edge $\{u,v\}$ and removing duplicate edges if necessary.

A digraph $D$ is \emph{weakly connected} if its underlying graph $D_u$ is connected. Vertex and arc deletion are defined analogously to the undirected case. Cut arcs, cut vertices, and biconnectivity for digraphs are defined according to whether weak connectivity is preserved after deletion.

The \emph{in-degree} $\mathrm{indeg}_D(v)$ and \emph{out-degree} $\mathrm{outdeg}_D(v)$ of a vertex $v$ are the number of arcs with head $v$ and with tail $v$, respectively.

A vertex $s$ with $\mathrm{indeg}_D(s) = 0$ is called a \emph{source} (or the \emph{root} when it is the only source of $D$), and a vertex $t$ with $\mathrm{outdeg}_D(t) = 0$ is called a \emph{sink}. A \emph{directed path} is a digraph $P$  with $V(P)=\{v_1,v_2, \dots,v_k\}$ and  $A(P)=\{(v_1, v_2),(v_2,v_3), \dots , (v_{k-1},v_k)\}$, where the vertices are all distinct. The vertices $v_1$ and $v_k$ are the \emph{first} and \emph{last} vertices of $P$, respectively. The length and internal vertices of $P$ are defined as in the undirected case.

A \emph{directed cycle} is digraph $C$ with $V(C)=\{v_1,v_2, \dots,v_k\}$ and $A(C)=\{(v_1, v_2),(v_2,v_3), \dots , (v_{k-1},v_k), (v_k, v_1)\}$ where the vertices are all distinct. A digraph is \emph{acyclic} if it has no directed cycle as a subgraph.
For an arc $(u,v)$ of an acyclic digraph, $u$ is sometimes called a \emph{parent} of $v$, and $v$ a \emph{child} of $u$.

\section{Known Results on Planarity}\label{sec:known_results}
\subsection{Planarity of Undirected Graphs}

The study of planarity in undirected graphs is a classical topic in graph theory. A \emph{drawing} of an undirected graph $G$ is an embedding of each vertex $v$ of $G$ to a point in the plane, and each edge $\{u,v\}$ to a simple curve connecting $u$ and $v$ without self-intersections. If $G$ can be drawn in the plane so that each edge crosses only at its endpoints, $G$ is said to be \emph{planar}, and such a drawing is called a \emph{planar drawing} or \emph{plane graph} of $G$.

Given a planar drawing of a planar graph $G$, a \emph{face} is a region of a planar drawing. The unbounded region is called the \emph{external face} while the others are \emph{interior faces}. 
A planar graph $G$ is \emph{outerplanar} if there exists a planar drawing of $G$ such that all vertices of $G$ are placed on the exterior face.

Planar graphs are characterized by forbidden structures, notably by the results of Kuratowski \cite{Kuratowski1930} and Wagner \cite{Wagner1937}.

\begin{theorem}[\cite{Kuratowski1930}]\label{thm:kuratowski}
An undirected graph $G$ is planar if and only if it contains no subgraph that is homeomorphic to $K_{3,3}$ or $K_5$.
\end{theorem}

Since $K_{3,3}$ and $K_5$ contain no vertices of degree two or less, Theorem~\ref{thm:kuratowski} can be rephrased as follows: An undirected graph $G$ is planar if and only if it contains no subgraph that is isomorphic to a subdivision of $K_{3,3}$ or $K_5$.

\begin{theorem}[\cite{Wagner1937}]\label{thm:wagneri}
An undirected graph $G$ is planar if and only if it contains neither $K_{3,3}$ nor $K_5$ as a minor.
\end{theorem}

Analogous to Kuratowski's theorem, the forbidden subgraphs of outerplanar graphs are also known well.

\begin{theorem}[\cite{AIHPB_1967__3_4_433_0}]\label{outer_forbidden}
An undirected graph $G$ is outerplanar if and only if it contains no subgraph that is homeomorphic to $K_{2,3}$ or $K_4$.
\end{theorem}

\subsection{Planarity for Directed Graphs}
A digraph $D$ is said to be \emph{planar} (resp. \emph{outerplanar}) if its underlying graph $D_u$ is planar (resp. outerplanar). The concept of ``upward planarity'' has long been studied in graph theory as a notion of planarity specific to digraphs   (e.g., \cite{HL1996upss,bbmt-oupsst-1998,DIBATTISTA1988175, garg1995upward}).

A digraph is \emph{upward planar} if it admits a planar upward drawing, where an \emph{upward drawing} of a digraph is such that all the edges are represented by directed curves increasing monotonically in the vertical direction \cite{bbmt-oupsst-1998}. It is known that a digraph has an upward drawing if and only if it is acyclic \cite{bbmt-oupsst-1998}.

Theorem~\ref{thm:up_pst} characterizes upward planar digraphs using planar $st$-digraphs (see Definition~\ref{dfn:st-digraph}). We will use this result in Section~\ref{sec:characterization.transformation}.

\begin{definition}\label{dfn:st-digraph}
An \emph{$st$-digraph} is any directed graph that has a unique source $s$, a unique sink $t$, and a unique arc $(s, t)$. 
In particular, an acyclic and planar $st$-digraph is called an \emph{acyclic planar} $st$-digraph.
\end{definition}

\begin{theorem}[\cite{DIBATTISTA1988175, Kelly1987}; see also Theorem 1 in \cite{garg1995upward}]\label{thm:up_pst}
A directed graph $G$ is upward planar if and only if there exists a planar $st$-digraph containing $G$ as a spanning subgraph.
\end{theorem}

\begin{figure}[htbp]
\centering
\includegraphics[width=0.5\textwidth]{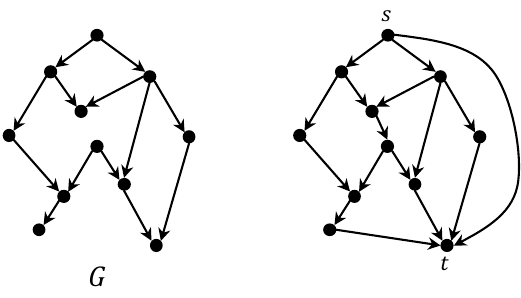}
\caption{An illustration of Theorem \ref{thm:up_pst}. Left: an upward planar digraph $G$. Right: an acyclic planar $st$-digraph that contains $G$ as a spanning subgraph.
\label{fig:ex_up_pst}
}
\end{figure}

\subsection{Terminal Planar Phylogenetic Networks}\label{sec:terminal.planar.literature}

Phylogenetic networks are classes of graphs, directed or undirected, that generalize phylogenetic trees and have been widely used to describe complex evolutionary histories and relationships among species. Typically, in a directed phylogenetic network, the leaves represent extant species, and the root represents their common evolutionary ancestor. 

Planar drawings (or nearly planar drawings) of phylogenetic networks can facilitate the interpretation of evolutionary data. In particular, when there exists a planar drawing in which all terminal vertices, i.e., the root and the leaves, are placed on the external face, even intricate reticulate evolution can be visually represented in a clear manner. Motivated by this, \upd{Moulton--Wu} \cite{WuMoulton9804871} introduced the notion of terminal planar networks. The terminal planar property is defined analogously for undirected and directed graphs, but in this subsection, we review the characterization of directed terminal planar networks as given in \cite{WuMoulton9804871}.

\begin{definition}\label{dfn:phylo}
Let $S = \{t_1, \dots, t_n\}$ be a nonempty finite set ($n \geq 1$). An \emph{undirected phylogenetic network} (with terminal set $S$) is defined to be a connected undirected graph $N$ such that $S$ can be identified with the set $\{v\in V(N) \mid \deg_N(v) = 1\}$. In this case, each element of $S \cap V(N)$ is  called a \emph{leaf} of $N$. 
A \emph{directed phylogenetic network} (with source $s$ and sink set $S$) is defined to be a connected acyclic digraph $N$ that satisfies the following conditions:
\begin{enumerate}
    \item $N$ has a unique source $s$ with $(\mathrm{indeg}_N(s),\mathrm{outdeg}_N(s)) = (0,1)$.
    \item $S$ can be identified with the set of sinks of $N$, and each sink $t_i$ of $N$ satisfies $(\mathrm{indeg}_N(t_i), \mathrm{outdeg}_N(t_i)) = (1,0)$.
\end{enumerate}
In a directed phylogenetic network $N$, the source $s$ and each sink $t_i$ are also called the \emph{root} and the \emph{leaf} of $N$, respectively.
\end{definition}

\begin{remark}\label{rem:no.null}
Any directed phylogenetic network $N$ with a source $s$ and sinks $t_1, \dots, t_k$ ($k \ge 1$) can be converted into an undirected phylogenetic network $N_u$ with terminal set $\{s, t_1, \dots, t_k\}$ by simply ignoring all edge orientations. However, for an undirected phylogenetic network with terminal set $S$, it may not be possible to assign edge directions so that the resulting directed phylogenetic network has one vertex in $S$ designated as the root and the remaining vertices in $S$ as leaves. 
This is illustrated by the undirected phylogenetic networks $G_1$ and $G_2$ in 
Figure~\ref{fig:rem-empty}. While $G_2$ can be oriented to obtain a directed phylogenetic network such that $\{t_1, t_2, t_3, t_4\}$ consists of a unique source and three sinks, $G_1$ cannot be oriented into a directed phylogenetic network with $\{t_1, t_2\}$ being a pair of the source and sink. 

To see the difference between $G_1$ and $G_2$, consider the block induced by $\{a,b,c\}$ of each graph. In $G_1$, regardless of which of $t_1$ and $t_2$ is the source \upd{and the other is the sink}, there is no pair $\{v_{\mathrm{in}}, v_{\mathrm{out}}\}$ of distinct vertices in $\{a,b,c\}$ such that the flow from the source comes into the block through $v_{\mathrm{in}}$ and goes out from $v_{\mathrm{out}}$ to arrive at \upd{the} sink in the end. More precisely, the flow from the source inevitably enters and exits $a$ before arriving at \upd{the} sink, implying that the block must contain a directed cycle \upd{on $\{a,b,c\}$}.  By contrast,   in $G_2$, when either $t_1$ or $t_2$ is the source \upd{and thus $t_3$ is a sink}, the flow enters $a$ and leaves through $b$ to arrive at $t_3$. Similarly, when $t_3$ is the source, the flow enters $b$ and leaves through $a$ to arrive at \upd{the sinks} $t_1$ \upd{and} $t_2$. \upd{In either case, one can easily construct a flow on $G_2$ that avoids directed cycles.}
\end{remark}

\begin{figure}[htbp]
\centering
\includegraphics[width=.8\textwidth]{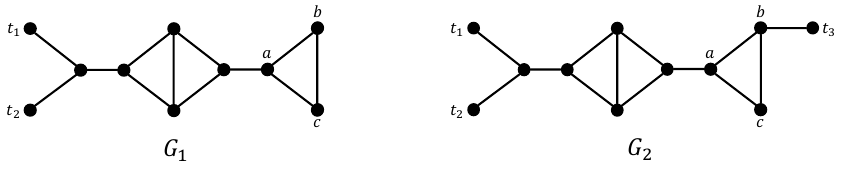}
\caption{\upd{The undirected phylogenetic networks $G_1$ and $G_2$ mentioned in Remark~\ref{rem:no.null}: $G_2$ can be oriented into a directed phylogenetic network while $G_1$ cannot.} 
\label{fig:rem-empty}
}
\end{figure}

\begin{definition}\label{dfn:terminal}
An undirected phylogenetic network $N$ is \emph{terminal planar} if there exists a planar drawing of $N$ in which all leaves are on the external face. Similarly, a directed phylogenetic network $N$ is \emph{terminal planar} if there is a planar drawing in which both the root and all leaves of $N$ are on the external face.
\end{definition}

By Definition~\ref{dfn:terminal}, if a directed phylogenetic network $N$ is terminal planar, then its underlying graph $N_u$ is also terminal planar. For an undirected phylogenetic network $N$ with terminal set $S$, suppose there exists a acyclic digraph $\vec{N}$ with a unique source $s\in S$ and a sink set $S\setminus \{s\}$, meaning that $\vec{N}$ is a directed phylogenetic network. Then,  if $N$ is terminal planar, so is $\vec{N}$.

\upd{Moulton--Wu} \cite{WuMoulton9804871}  obtained  Theorem~\ref{thm:ptpu} that  characterizes when a directed phylogenetic network $N$ is terminal planar using a supergraph $N^+$ of $N$, which is defined in Definition~\ref{dfn:N^+} (see Figure~\ref{fig:ex_t_completion} for an illustration). 

\begin{definition}[\cite{WuMoulton9804871}]\label{dfn:N^+}
Given a directed phylogenetic network $N = (V, A)$ with source $s$ and sink set $\{t_1, \dots, t_n\}$, let $V(N^+) := V \cup \{t\}$ and $A(N^+) := A \cup \{(t_1, t), \dots, (t_n, t)\}$. The digraph $N^+$ is called the \emph{$t$-completion} of $N$.
\end{definition}

Note that $N^+$ was called a ``completion'' in \cite{WuMoulton9804871}, but in this paper, we will use the term ``$t$-completion'' to distinguish it from other completions to be defined later.

\begin{theorem}[\cite{WuMoulton9804871}]\label{thm:ptpu}
For a directed phylogenetic network $N$, $N$ is terminal planar if and only if its $t$-completion $N^+$ is upward planar.
\end{theorem}

\upd{Moulton--Wu} \cite[Theorem 3.2]{WuMoulton9804871} proved the following strict inclusion relations among various classes of planar phylogenetic networks:
\begin{equation}\label{eq:inclusion}
\mathcal{P}_{\mathrm{outer}}(S) \subsetneq \mathcal{P}_{\mathrm{terminal}}(S) \subsetneq \mathcal{P}_{\mathrm{upward}}(S) \subsetneq \mathcal{P}(S).
\end{equation}
Here, for any set $S$ with $|S| \ge 2$, each of the sets $\mathcal{P}(S)$, $\mathcal{P}_{\mathrm{upward}}(S)$, $\mathcal{P}_{\mathrm{terminal}}(S)$, and $\mathcal{P}_{\mathrm{outer}}(S)$ denotes the collection of all planar, upward planar, terminal planar, and outerplanar directed phylogenetic networks with sink set $S$, respectively. See Figure~\ref{fig:ex_phylonet_planarity} for explicit examples illustrating that each inclusion in \eqref{eq:inclusion} is strict.

The first inclusion in \eqref{eq:inclusion} follows from the fact that if a directed (or undirected) phylogenetic network $N$ is outerplanar, then there exists a planar drawing in which all vertices of $N$ are placed on the external face, namely, $N$ is terminal planar. The second inclusion in \eqref{eq:inclusion} follows from Theorem~\ref{thm:ptpu} since $N$ is a subgraph of $N^+$. 

\begin{figure}[htbp]
\centering
\includegraphics[width=.5\textwidth]{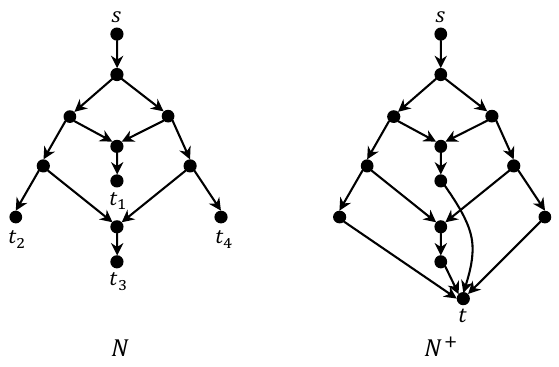}
\caption{
A directed phylogenetic network $N$ with source $s$ and sink set $\{t_1, t_2, t_3, t_4\}$ (left) and the $t$-completion $N ^+$  of $N$ (right).
\label{fig:ex_t_completion}
}
\end{figure}

\begin{figure}[htbp]
\centering
\includegraphics[width=.7\textwidth]{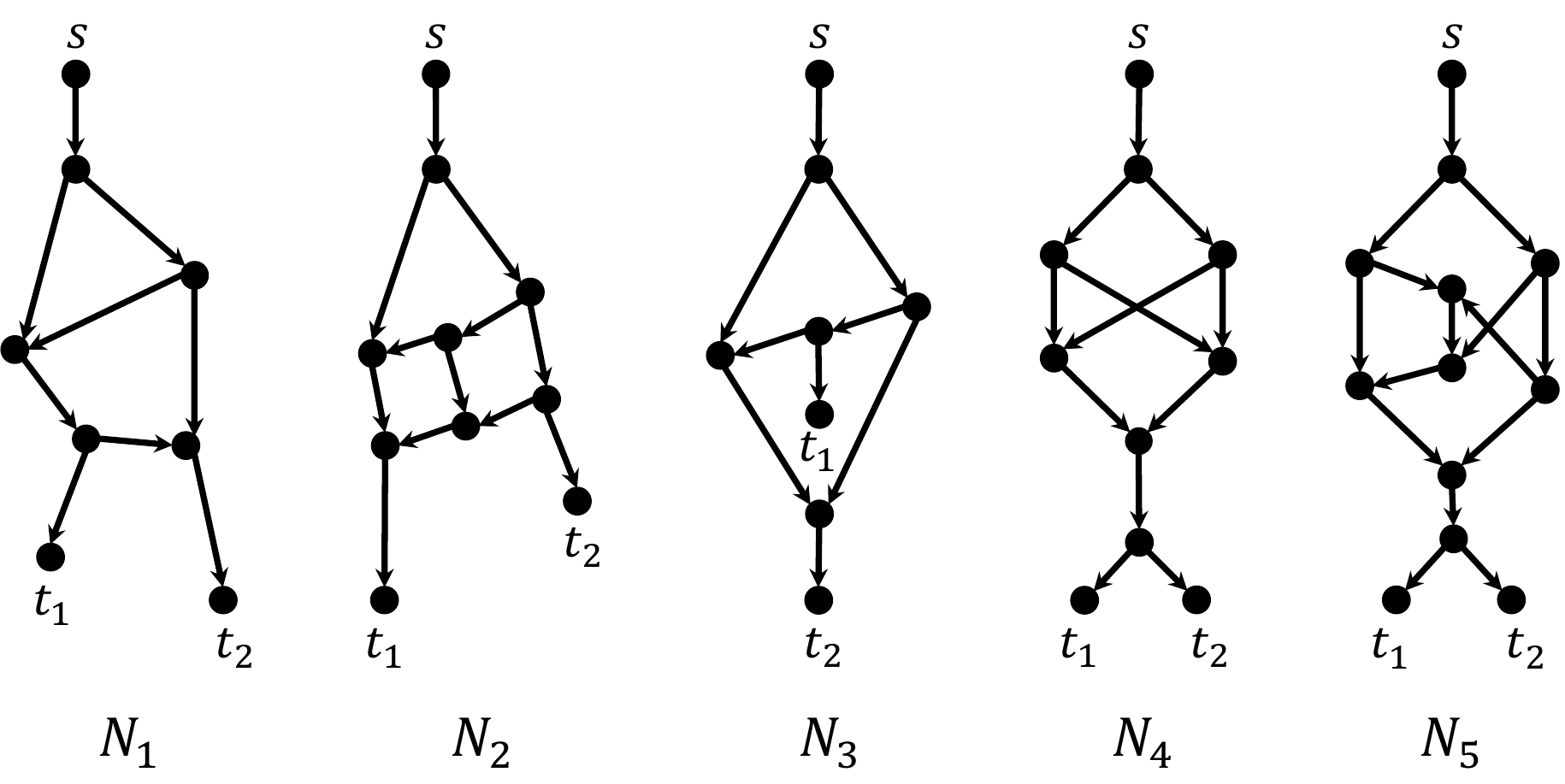}
\caption{Examples demonstrating the strict inclusions expressed in \eqref{eq:inclusion}. $N_1\in \mathcal{P}_{\mathrm{outer}}(S)$, $N_2\in \mathcal{P}_{\mathrm{terminal}}(S)\setminus \mathcal{P}_{\mathrm{outer}}(S)$, $N_3\in \mathcal{P}_{\mathrm{upward}}(S)\setminus \mathcal{P}_{\mathrm{terminal}}(S)$, $N_4\in \mathcal{P}(S)\setminus \mathcal{P}_{\mathrm{upward}}(S)$, $N_5 \not \in \mathcal{P}(S)$ (all examples except $N_1$ are reproduced from Fig.\ 2 of \cite{WuMoulton9804871}).
\label{fig:ex_phylonet_planarity}
}
\end{figure}

\section{Characterization of Terminal Planar Networks Using Supergraphs}\label{sec:characterization.transformation}

Theorem~\ref{thm:ptpu} provides a characterization of terminal planar directed phylogenetic networks by focusing on the upward planarity of the $t$-completion. However, since the forbidden structures for upward planar networks remain unresolved, Theorem~\ref{thm:ptpu} does not directly lead to a characterization of the forbidden structures for terminal planar networks.

In this section, we introduce new supergraphs different from the $t$-completion and provide a new characterization of terminal planar directed  phylogenetic networks by focusing on the planarity of such supergraphs (\upd{Theorem~\ref{lem:N_completion}}).

\begin{definition}\label{dfn:st-completions}
Let $N = (V, A)$ be a directed phylogenetic network with source $s$ and sink set $\{t_1,\dots, t_n\}$. Let $t$ be a new vertex not contained in $V$. Two supergraphs of $N$, denoted by $N^\ast$ and $N^x$,  are defined as follows:
\begin{enumerate}
    \item $N^\ast$ is the acyclic $st$-digraph with $V(N^\ast) := V \cup \{t\}$ and $A(N^\ast) := A \cup \{(t_1, t), \dots, (t_n, t)\} \cup \{(s, t)\}$. We call $N^\ast$ the \emph{$st$-completion} of $N$.
    \item $N^x$ is \upd{the} acyclic digraph (without multiple edges) with $V(N^x) := V \cup \{t\}$ and  $A(N^x) := A \cup \{(\tilde{t_1}, t), \dots, (\tilde{t_n}, t)\} \cup \{(\tilde{s}, t)\}$, where $\tilde{s}$ denotes the unique child of $s$ in $N$, and for each $i \in [1, n]$, $\tilde{t_i}$ denotes the unique parent of $t_i$ in $N$. We call $N^x$ the \emph{terminal cut-completion} of $N$.
\end{enumerate}
\end{definition}
\upd{
\begin{theorem}\label{lem:N_completion}
For any directed phylogenetic network $N$, the following statements are equivalent. 
\begin{enumerate}
    \item $N$ is terminal planar.
    \item The $st$-completion $N^\ast$ of $N$ is planar.
    \item The terminal cut-completion $N^x$ of $N$ is planar.
\end{enumerate}
\end{theorem}
}
\begin{proof}
First, we prove the equivalence of 1 and 2. Suppose $N$ is terminal planar. By Theorem~\ref{thm:ptpu}, its $t$-completion $N^+$ is upward planar. Theorem~\ref{thm:up_pst} then ensures the existence of an acyclic planar $st$-digraph $D$ for which $N^+$ is a spanning subgraph.
To see that $N^\ast$ is a subgraph of $D$, consider the following. By Definition~\ref{dfn:st-completions}, $V(N^\ast) = V(N^+) = V(D)$ and  $A(N^\ast) = A(N^+) \cup \{(s, t)\}$ hold. Since $(s, t) \in A(D)$ by Definition~\ref{dfn:st-digraph}, together with $A(N^+) \subseteq A(D)$, it follows that $A(N^\ast) \subseteq A(D)$. Therefore, $N^\ast$ is a subgraph of $D$, and thus $N^\ast$ is planar.
Conversely, if $N^\ast$ is planar, it is a spanning subgraph of itself, so by Theorem~\ref{thm:up_pst}, $N^\ast$ is upward planar. Since by definition $N^+ \subseteq N^\ast$, $N^+$ is also upward planar, and by Theorem~\ref{thm:ptpu}, $N$ is terminal planar.

Next, we show the equivalence of 2 and 3. By Theorem~\ref{thm:kuratowski}, it suffices to show that the underlying graph $N^\ast_u$  of $N^\ast$ contains a subgraph homeomorphic to $K_5$ or $K_{3,3}$ if and only if $N^x_u$ does as well.
Note that in $N^\ast_u$, the source $s$ and each sink $t_i$ have degree-2. Let $\mathrm{sm}(N^\ast_u)$ denote the simple graph obtained by smoothing all terminal vertices $s, t_1, \dots, t_n$ of $N^\ast_u$. Then, we can see that $N^x_u - \{s, t_1, \dots, t_n\}$ is isomorphic to $\mathrm{sm}(N^\ast_u)$. Thus, if $N^\ast_u$ has a subgraph homeomorphic to  $K_5$, then $N^x_u$ also has the same property. 
Conversely, if $N^x_u$ contains a subgraph $G$ that is homeomorphic to $K_5$, then $G \subseteq N^x_u - \{s, t_1, \dots, t_n\}$ since $K_5$ has no vertex of degree one. Thus,  $\mathrm{sm}(N^\ast_u)$ also contains $G$ as a subgraph, which implies that $N^\ast_u$ contains a subgraph homeomorphic to $K_5$. The same argument applies for $K_{3,3}$. Thus, the equivalence of 2 and 3 is established.
\end{proof}

We note that the notion of terminal planarity is related to bar-visibility graphs and digraphs \cite{FENG2025342}, which is sometimes called $\epsilon$-visibility \cite{TamassiaTollis1986} or left-visibility \cite{Wismath1985}. It is known that a (possibly multi-source and multi-sink) digraph $D$ is a bar-visibility digraph if and only if $D^\prime$ is an acyclic planar $st$-digraph \cite{TamassiaTollis1986,Wismath1985,Wismath_1989}, where $D^\prime$ is a supergraph obtained from $D$ by introducing a unique source $s$, a unique sink $t$ and the arc $(s,t)$. In our setting, if $D$ is a directed phylogenetic network $N$, then $D^\prime$ can be seen as essentially the same as the $st$-completion $N^\ast$, except for the difference that $N^\ast$ has an arc $(s,\tilde{s})$ while $D^\prime$ has a directed path of length $2$ from $s$ to $\tilde{s}$.  Then, by Theorem~\ref{lem:N_completion}, we can see that a directed phylogenetic network $N$ is terminal planar if and only if $N$ is a bar-visibility digraph. The interested reader is referred to \cite{Wismath1985}.

\section{Characterization of Terminal Planar Networks by Forbidden Subgraphs}\label{sec:characterization.forbidden}

In Sections~\ref{sec:terminal.planar.literature} and \ref{sec:characterization.transformation}, we discussed the necessary and sufficient conditions for a directed phylogenetic network to be terminal planar, focusing on the upward planarity or planarity of their supergraphs.  Hereafter, we will consider the forbidden subgraphs that characterize terminal planarity in directed and undirected phylogenetic networks.
The main result of this paper, Theorem~\ref{thm:main} is an analogue of Theorem~\ref{thm:kuratowski}. 
To state Theorem \ref{thm:main}, we first introduce the concepts of \emph{$0/1$-labeled graphs} and \emph{cut-labeled graphs}, along with basic graph operations defined on such graphs.

\subsection{$0/1$-Labeled Graphs}
A \emph{$0/1$-labeled graph} is a pair $(G, g)$, where $G = (V, E)$ is a connected undirected graph and $g: V \cup E \to \{0, 1\}$ is a function. In this setting, $(G, g)$ is called a \emph{$0/1$-labeled graph of $G$ determined by $g$}, and $g$ is referred to as the \emph{labeling function} of $G$. Each vertex $v \in V$ and each edge $e \in E$ are called a \emph{vertex} and an \emph{edge} of $(G, g)$, respectively. For each element $x \in V \cup E$, we say $x$ has \emph{label 1} if $g(x) = 1$, and \emph{label 0} otherwise.

For two $0/1$-labeled graphs $(F, f)$ and $(G, g)$, we say that $(F,f)$ and $(G,g)$ are \emph{label-preserving isomorphic} if $F \simeq G$ and each vertex (resp.\ edge) of $F$ has the same label of its corresponding vertex (resp.\ edge) of $G$.  
If $F \subseteq G$ and $f$ is the restriction of $g$ to $V(F) \cup E(F)$, then $(F, f)$ is called a \emph{label-preserving subgraph} of $(G, g)$, in which case we write $(F, f) \subseteqq (G, g)$ and often represent $(F,f)$ by $(F, g|_F)$. 
If $G \subseteq F$ and $f$ is an extension of $g$ to $V(F) \cup E(F)$, then $(F, f)$ is called a \emph{label-preserving supergraph} of $(G, g)$, denoted by $(G, g) \subseteqq (F, f)$.

\begin{definition}\label{dfn:label.smoothing}
A vertex $v$ of a $0/1$-labeled graph $(G,g)$ is called \emph{label-preserving smoothable} if $\deg_G(v) = 2$ and the two edges incident to $v$, $e_1 = \{u_1,v\}$ and $e_2 = \{u_2,v\}$, satisfy $g(e_1) = g(e_2)$. For a label-preserving smoothable vertex $v$ of $(G,g)$, \emph{label-preserving smoothing of vertex $v$} is defined as obtaining a $0/1$-labeled graph $(H,h)$ from $(G,g)$ as follows:
\begin{itemize}
\item $V(H) := V(G) \setminus \{v\}$, \quad $E(H) := (E(G) \setminus \{e_1, e_2\}) \cup \{\{u_1,u_2\}\}$
\item For each $x \in V(H) \cup E(H)$, define 
$
h(x) := \begin{cases}
g(e_1) & \text{if } x = \{u_1,u_2\}, \\
g(x) & \text{otherwise}.
\end{cases}
$
\end{itemize}
Moreover, the $0/1$-labeled graph obtained by label-preserving smoothing all possible vertices of $(G,g)$ is called the \emph{label-preserving smoothed graph} of $(G,g)$.
\end{definition}

\begin{definition}\label{dfn:label.subdiv}
For an edge $\{u,v\}$ of a $0/1$-labeled graph $(G,g)$, \emph{label-preserving subdivision} of the edge $\{u,v\}$ is defined as obtaining a $0/1$-labeled graph $(H,h)$ from $(G,g)$. The graph $H$ is obtained by replacing the edge $\{u,v\}$ of $G$ with a path $[u,v]$ of length at least one. The labeling function $h$ of $H$ is defined as follows:
\begin{itemize}
    \item For each edge or internal vertex $x$ (with $x \neq u,v$) on the path $[u,v]$ in $H$ corresponding to the edge $\{u,v\}$, set $h(x) := g(\{u,v\})$.
    \item For each $x \in V(G) \cup E(G)$, set $h(x) := g(x)$.
\end{itemize}
Given two $0/1$-labeled graphs $(G,g)$ and $(H,h)$, if $(H,h)$ can be obtained from $(G,g)$ by applying a finite sequence of label-preserving edge subdivisions, then $(H,h)$ is called a \emph{label-preserving subdivision} of $(G,g)$.
\end{definition}

\begin{definition}\label{dfn:label.homeomorphism}
Given two $0/1$-labeled graphs $(G,g)$ and $(H,h)$, let $(G^{\prime}, g^{\prime})$ and $(H^{\prime}, h^{\prime})$ be $0/1$-labeled graphs obtained from $(G,g)$ and $(H,h)$, respectively, by finite sequences of label-preserving edge subdivisions and label-preserving vertex smoothings. If $(G^{\prime}, g^{\prime}) \cong (H^{\prime}, h^{\prime})$, then $(G,g)$ and $(H,h)$ are said to be \emph{label-preserving homeomorphic}. We write $(G,g) \approxeq (H,h)$.
\end{definition}

Vertex deletion and edge deletion in $0/1$-labeled graphs are defined analogously to those in unlabeled graphs.
For an edge $\{u,v\}$ of a $0/1$-labeled graph $(G,g)$, the \emph{label-preserving contraction} of the edge $\{u,v\}$ is defined as obtaining a $0/1$-labeled graph $(F,f)$ from $(G,g)$ as follows: $F$ is the graph obtained by contracting the edge $\{u,v\}$ in $G$. The label of the vertex $w \in V(F)$, created by identifying $u$ and $v$, is given by  
$f(w) := \max\{ g(u), g(v) \}$,
while the labels of all other vertices and edges in $F$, except for $w$, are the same as those of their corresponding elements in $G$.

\subsection{Cut-Labeled Graphs}
For a connected undirected graph $G = (V, E)$, the $0/1$-labeled graph $(G, \ell)$ with the labeling function $\ell: V \cup E \rightarrow \{0, 1\}$ defined in equation~\eqref{eq:cut} is called the \emph{cut-labeled graph} of $G$ and is written as $L(G)$.
\begin{equation}\label{eq:cut}
\ell(x) := 
\begin{cases}
1 & \text{if } x \in V \cup E \text{ is a cut vertex or cut edge of } G, \\
0 & \text{otherwise.}
\end{cases}
\end{equation}

\begin{definition}\label{dfn:cut_completion}
Let $N = (V, A)$ be a directed phylogenetic network with a source $s$ and a sink set $\{t_1, \dots, t_n\}$. The \emph{cut-completion} of $N$, denoted by $N^c$, is the acyclic digraph defined by 
$V(N^c) := V \cup \{t\}, \quad A(N^c) := A \cup \{(v_1, t), \dots, (v_k, t)\}$,
where $\{v_1, \dots, v_k\}$ is the set of all cut vertices of the underlying graph $N_u$ of $N$. 
Let $L(N_u)$ denote the cut-labeled graph of the underlying graph $N_u = (V, E)$ of $N$. Let $N_u^c = (V^c, E^c)$ be the underlying graph of the cut-completion $N^c$.
The $0/1$-labeled graph $L^c(N_u) := (N_u^c, \ell^c)$ is called the \emph{label-preserving cut-completion} of $L(N_u)$, where $\ell^c$ is a function on $V^c \cup E^c$ defined by
\begin{equation}\label{eq:label.cut}
\ell^c(x) := 
\begin{cases}
\ell(x) & \text{if } x \in V \cup E, \\
0 & \text{otherwise.}
\end{cases}
\end{equation}
\end{definition}

In Proposition~\ref{prop:relation_completion}, recall that the $st$-completion $N^*$ is defined in Definition~\ref{dfn:st-completions} and the terminal cut-completion $N^x$ in Definition~\ref{dfn:cut_completion}.

\begin{proposition}\label{prop:relation_completion}
For any directed phylogenetic network $N$, if the cut-completion $N^c$ is planar, then both $N^x$ and $N^\ast$ are planar. 
\end{proposition}

\begin{proof}
Since $N^x \subseteq N^c$, the planarity of $N^c$ immediately implies that of $N^x$. Furthermore, if $N^x$ is planar, then by \upd{Theorem~\ref{lem:N_completion}}, the graph $N^\ast$ is also planar. This completes the proof.
\end{proof}

\begin{figure}[htbp]
\centering
\includegraphics[width=0.95\textwidth]{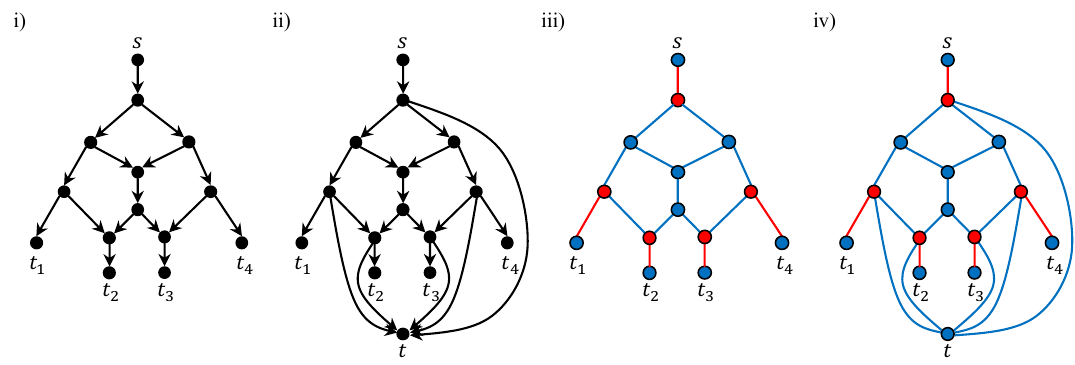}
\caption{Illustration of cut-labeled graphs and Definition~\ref{dfn:cut_completion}. i) A directed phylogenetic network $N$ with source $s$ and sink set $\{t_1, t_2, t_3,t_4\}$. ii) The cut-completion $N^c$ of $N$. iii) The cut-labeled graph $L(N_u)$ of $N$. iv) The label-preserving cut-completion $L^c(N_u)$ of $L(N_u)$.
\label{fig:cut.completion}
}
\end{figure}

\subsection{$\mathcal{H}_i$ Structures}

Let $K_{3,3}^{-}$ denote the graph obtained by removing an arbitrary edge from $K_{3,3}$. 
Let $K_{3,3}^{(-e)}$ denote the graph obtained by subdividing an arbitrary edge of $K_{3,3}$ exactly twice and then removing the edge between the two vertices created by this subdivision. 
Let $K_{3,3}^{(-v)}$ denote the graph obtained by subdividing each edge of $K_{3,3}$ at least once, removing an arbitrary vertex of $K_{3,3}$, and then smoothing all smoothable vertices.
The graphs $K_5^{-}$, $K_5^{(-e)}$, and $K_5^{(-v)}$ are defined analogously by replacing $K_{3,3}$ with $K_5$ in the above definitions.

\begin{definition}\label{dfn:graph.family}
For each $i \in [1,6]$, we define the undirected graph $H_i$ by 
$H_1 \simeq K_{3,3},\ H_2 \simeq K_{3,3}^{(-v)},\ H_3 \simeq K_{3,3}^{(-e)},\ H_4 \simeq K_5,\ H_5 \simeq K_5^{(-v)},\ H_6 \simeq K_5^{(-e)}$.
An \emph{$\mathcal{H}_i$ graph} is any $0/1$-labeled graph obtained from a $0/1$-labeled graph $(H_i, h_i)$ satisfying the following conditions by label-preserving contraction of zero or more pendant edges:

\begin{enumerate}\addtolength{\itemsep}{1.5ex}
\item $\forall v \in V(H_i)$, 
$h_i(v) := \begin{cases}
1 & \text{if } \deg_{H_i}(v) = 1 \\
0 & \text{if } \exists x \in \mathrm{Adj}_{H_i}(v) \text{ with } \deg_{H_i}(x) = 1
\end{cases}$

\item $\forall e \in E(H_i)$, $h_i(e) := 0$.
\end{enumerate}
\end{definition}

The number of $\mathcal{H}_i$ graphs in Definition \ref{dfn:graph.family} is easily determined. For example, the number of $\mathcal{H}_5$ graphs equals $5$, since by symmetry it equals the number of distinct ways in which pendant edges can be contracted. 
Similarly, it is easy to verify that there are $12$ $\mathcal{H}_2$ graphs.

We also note that it is possible that a 0/1-labeled graph has an $\mathcal{H}_i$ graph and \upd{an} $\mathcal{H}_j$ graph  ($i\neq j$) as its \upd{label-preserving} subgraphs. 
For example, if a graph contains an $\mathcal{H}_1$ graph \upd{where all vertices have label $1$}, then it \upd{contains} an $\mathcal{H}_2$ graph that is obtained by \upd{the label-preserving contraction of the three} pendant edges of $(H_2, h_2)$ in Figure \ref{fig:forbidden}.

\begin{figure}[htbp]
\centering
\includegraphics[width=.9\textwidth]{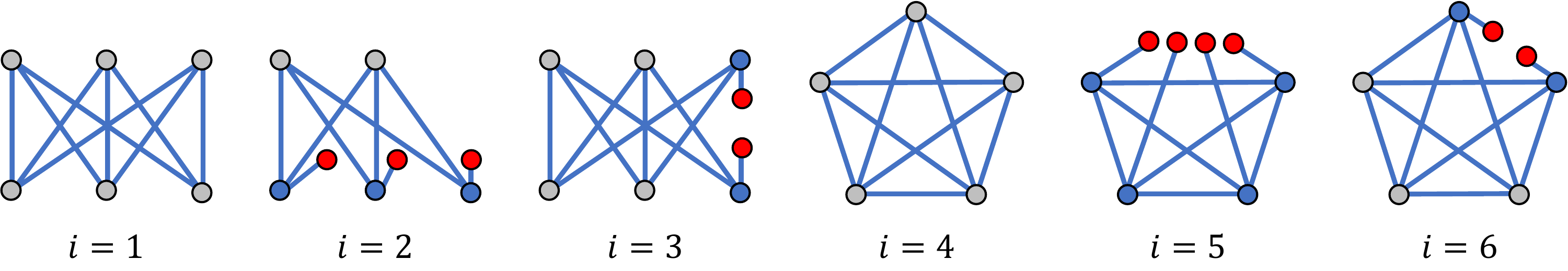}
\caption{
Illustration of the family of $(H_i, h_i)$ satisfying the conditions in Definition~\ref{dfn:graph.family}. Vertices labeled 1 are shown in red; vertices with unspecified labels are shown in gray. All edges are labeled 0 and shown in blue.
\label{fig:forbidden}
}
\end{figure}

\begin{definition}\label{def:labeled_homeomorphism}
Let $(G,g)$ be a $0/1$-labeled graph. Suppose there exist $i \in [1,6]$ and a label-preserving subgraph $(F, g|_F)$ of $(G, g)$ such that $(F, g|_F)$ is label-preserving homeomorphic to an $\mathcal{H}_i$ graph. Then $(G, g)$ is said to \emph{contain} an $\mathcal{H}_i$ structure, and $(F, g|_F)$ is called an \emph{$\mathcal{H}_i$ structure} of $(G, g)$. If no such structure exists, $(G, g)$ is said to \emph{not contain} an $\mathcal{H}_i$ structure.
An $\mathcal{H}_i$ structure $(F, g|_F)$ of $(G, g)$ is called \emph{minimal} if there exists no proper label-preserving subgraph $(F^\prime, g|_{F^\prime})$ of $(G,g)$ with $F^\prime \subsetneq F$ that is also an $\mathcal{H}_i$ structure of $(G, g)$.
\end{definition}

\begin{proposition}\label{prop.map.phi}
Let $(G,g)$ be a $0/1$-labeled graph, let $i \in [1,6]$ be arbitrary, and let $(H,h)$ be an $\mathcal{H}_i$ graph. Suppose that $(G,g)$ contains an $\mathcal{H}_i$ structure $(F, g|_F)$ that is label-preserving homeomorphic to $(H,h)$. Then there exists an injection $\psi \colon V(H) \to V(F)$ such that $(H,h)$ is label-preserving isomorphic to the graph obtained from $(F, g|_F)$ by smoothing all vertices outside the image $\{\psi(v) \mid v \in V(H)\}$.
\end{proposition}

\begin{proof}
The proposition states that $(F, g|_F)$ can be transformed into $(H, h)$ by label-preserving smoothing of finitely many vertices in $(F, g|_F)$, without subdividing any edges.
If $(H, h)$ contains no vertex $v$ with $\deg_H(v) = 2$, then there is no need to create new degree-2 vertices when transforming $(F, g|_F)$ into $(H, h)$, so indeed no edge subdivision is required, and $(F, g|_F)$ can be transformed into $(H, h)$ using only label-preserving smoothings.
Let $S \subseteq V(F)$ denote the set of vertices that are not smoothed in the transformation from $(F, g|_F)$ to $(H, h)$.
Then there exists a bijection $\tau$ from $S$ to $V(H)$, and $\psi := \tau^{-1}$ gives the desired injection.
Indeed, the image of $\psi$ coincides with $S$, and all vertices outside $S$ are smoothed, so the conditions of the proposition are satisfied.

Next, suppose $(H, h)$ has a vertex $v$ with $\deg_H(v) = 2$. 
In this case, $i \in \{2, 3\}$ and $(H, h)$ is an $\mathcal{H}_i$ graph obtained by label-preserving contraction of one or more pendant edges of $(H_i, h_i)$ (as in Definition~\ref{dfn:graph.family}), where $v$ is the contracted vertex where the ends of a pendant edge have been identified.
By Definition~\ref{dfn:graph.family}, each edge of $(H, h)$ has label $0$, and by assumption, $(F, g|_F) \approxeq (H, h)$. Note that by Definitions~\ref{dfn:label.smoothing} and ~\ref{dfn:label.subdiv},  if a $0/1$-labeled graph has an edge labeled $1$, then any graph label-preserving homeomorphic to it also has an edge labeled $1$.
Thus, every edge of $(F, g|_F)$ has label $0$, and applying any sequence of label-preserving subdivisions or smoothings to $(F, g|_F)$ yields graphs in which every edge also has label $0$.
Moreover, by Definition~\ref{dfn:label.subdiv}, every degree-2 vertex obtained by label-preserving subdivision of an edge labeled $0$ also has label $0$.
Suppose for contradiction that for some $v$ with $\deg_H(v) = 2$, there is no corresponding vertex in $(F, g|_F)$. This means that $v$ arises as a result of a label-preserving subdivision of $(F, g|_F)$ during the process of transforming from $(F, g|_F)$ to $(H, h)$. Then $h(v) = 0$ must hold. 
However, by the construction in Definition~\ref{dfn:graph.family}, such a vertex $v$ is obtained by label-preserving contraction of a pendant edge in $(H_i, h_i)$, which implies $h(v) = 1$. 
This is a contradiction. Therefore, for every vertex $v$ of $(H,h)$ with $\deg_H(v) = 2$, there exists a vertex of $(F, g|_F)$ that corresponds to $v$.
Consequently, as in the previous case, if $S \subseteq V(F)$ is the set of vertices not smoothed in the transformation from $(F, g|_F)$ to $(H, h)$, then the inverse $\psi := \tau^{-1}$ of the bijection $\tau: S \to V(H)$ provides the desired injection.
\end{proof}

The injection $\psi$ constructed in Proposition~\ref{prop.map.phi} is called an \emph{$\mathcal{H}_i$ structure mapping}. In general, $\psi$ is not unique. For example, since an $\mathcal{H}_1$ graph has no vertices with specified $0/1$ labels and all edges are labeled $0$, the mapping from vertices of the $\mathcal{H}_1$ graph to vertices of the $\mathcal{H}_i$ structure is not uniquely determined.
\subsection{Forbidden Structures for Terminal Planar Networks: Proof of Necessity}

\begin{proposition}\label{prop:cut.vertex.cycle}
Let $G$ be an undirected graph containing a path $P = u, \{u, v\}, v, \{v, w\}, w$ of length 2. If $v$ is not a cut vertex of $G$ and neither $\{u, v\}$ nor $\{v, w\}$ is a cut edge of $G$, then $G$ contains a cycle that includes $P$.
\end{proposition}

\begin{proof}
Since $v$ is not a cut vertex of $G$, $G - v$ is connected. Therefore, in $G$, there exists a $u$-$w$ path $P'$ that does not contain $v$ nor the edges $\{u, v\}$ and $\{v, w\}$. Thus, $G$ contains the cycle formed by $P \cup P'$.
\end{proof}

\begin{proposition}\label{prop:F}
Let $i \in [1,6]$, $(H, h)$ be any $\mathcal{H}_i$ graph, and suppose that the cut-labeled graph $(N_u, \ell)$ of a directed phylogenetic network $N$ contains a minimal $\mathcal{H}_i$ structure $(F, \ell|_F) \approxeq (H, h)$. Let $\psi: V(H) \to V(F)$ be the $\mathcal{H}_i$ structure mapping as defined in Proposition~\ref{prop.map.phi}. Then the following hold:

\begin{enumerate}
  \item For any vertex $v$ of $(H, h)$, $h(v) = 1$ if and only if $\psi(v)$ is a cut vertex of $N_u$.
  \item $F$ is contained in a single block of $N_u$.
\end{enumerate}
\end{proposition}

\begin{proof}
By Proposition~\ref{prop.map.phi}, the corresponding vertex $\psi(v)$ of $(F, \ell|_F)$ for each $v$ in $(H, h)$ satisfies $\ell|_F(\psi(v)) = h(v)$. Since $\ell|_F$ is the restriction of $\ell$ to $V(F) \cup E(F)$, we have $\ell|_F(\psi(v)) = \ell(\psi(v))$. Thus, $\ell(\psi(v)) = 1$ if and only if $h(v) = 1$. Noting that $\ell(\psi(v)) = 1$ exactly when $\psi(v)$ is a cut vertex of $N_u$. 

Recall that $(H, h)$ is obtained by label-preserving contraction of at least zero pendant edges of $(H_i, h_i)$ in Definition~\ref{dfn:graph.family}, so it may not contain any vertex $v$ with $\deg_H(v) = 1$. In this case, for any $v \in V(H)$, $H - v$ is connected, i.e., $H$ is 2-connected by Definition~\ref{dfn:graph.family}. Since $F \approx H$, $F$ is also 2-connected.

Next, suppose $(H, h)$ has exactly one vertex $v$ with $\deg_H(v) = 1$. Let $u$ denote the neighbor of $v$ in $(H, h)$. Then $H$ is the union of the 2-connected graph $H - v$ and the pendant edge $\{u, v\}$. Thus, $F$ consists of a 2-connected graph $F_0$ homeomorphic to $H - v$, and a path $[\psi(u), \psi(v)]_F$ homeomorphic to $\{u, v\}$.
Since any 2-connected subgraph of an undirected graph $G$ is contained in a single block of $G$, it suffices to show that $F$ is contained in some 2-connected subgraph of $N_u$.
For the path $[\psi(u), \psi(v)]_F$ of length $k \geq 1$, set $w_1 := \psi(u)$, $w_{k+1} := \psi(v)$, so $[\psi(u), \psi(v)]_F$ can be written as $w_1 - w_2 - \dots - w_k - w_{k+1}$. Adding the edge $\{w_0, w_1\} \in E(F_0)$, the path $w_0 - w_1 - w_2 - \dots - w_k - w_{k+1}$ is denoted $P_k$.
As discussed in the proof of Proposition~\ref{prop.map.phi}, since every edge of a $\mathcal{H}_i$ graph is labeled $0$, every edge of a $\mathcal{H}_i$ structure is also labeled $0$.
Thus, no edge of $P_k$ is a cut edge of $N_u$. Similarly, no internal vertex of $P_k$ is a cut vertex of $N_u$. This follows because, by Definition~\ref{dfn:graph.family}, $h(u) = 0$, so by the claim proved above, the corresponding vertex $\psi(u) = w_1$ in $(F, \ell|_F)$ is not a cut vertex of $N_u$, and by minimality of $(F, \ell|_F)$, each $w_j$ for $j \in [2, k]$ is not a cut vertex of $N_u$.
By Proposition~\ref{prop:cut.vertex.cycle}, for each $j \in [1, k]$, there exists a cycle $C_j \subseteq N_u$ containing $w_{j-1} - w_j - w_{j+1}$. Since $F_0$ is 2-connected and $C_1$ shares the edge $\{w_0, w_1\}$ with $F_0$, it is clear that, for any two vertices $x, y$ in $F_0 \cup C_1$, there exist two $x$-$y$ paths with no common internal vertices. Hence, $F_0 \cup C_1$ is 2-connected. By induction, $F_0 \cup C_1 \cup \dots \cup C_k$ is also 2-connected. Since $P_k \subseteq C_1 \cup \dots \cup C_k$, $F \subseteq F_0 \cup C_1 \cup \dots \cup C_k \subseteq N_u$, so there exists a block of $N_u$ containing $F$. The same conclusion holds when $(H, h)$ has more than one vertex $v$ with $\deg_H(v) = 1$ by a similar argument.
\end{proof}

\begin{lemma}\label{lem:disjoint_path}
Let $N_u$ be the underlying graph of a directed phylogenetic network $N$ with source $s$ and sink set $\{t_1, \dots, t_n\}$ and assume that $N$ has at least two edges. Let $G$ be a subgraph of $N_u$ that contains a cut vertex of $N_u$ and is contained within a block $B$ of $N_u$. For any subset $V_1 = \{v_1, \dots, v_k\} \ (k \geq 1)$ of the set of cut vertices of $N_u$ contained in $V(G)$, let $N^x$ and its sink $t$ be defined as in Definition~\ref{dfn:st-completions}. Then, the underlying graph $N^x_u$ of $N^x$ contains a collection $\{P_1, \dots, P_k\}$ of undirected paths satisfying the following conditions:
\begin{enumerate}
   \item For every $i \in [1, k]$, $P_i$ is a path in $N^x_u$ connecting $v_i \in V_1$ to the sink $t$.
   \item For any distinct $i, j \in [1, k]$, $V(P_i) \cap V(P_j) = \{t\}$.
   \item For every $i \in [1, k]$, $V(P_i) \cap V(B) = \{v_i\}$.
\end{enumerate}
\end{lemma}

\begin{proof}
For any connected graph $H$, the graph $T_H$ defined as follows is a tree~\cite[Theorem 1]{harary-prins1966}: $V(T_H)$ is the union of the set $\mathcal{A}$ of cut vertices of $H$ and the set $\mathcal{B}$ of blocks of $H$, and $\{a, b\} \in E(T_H)$ if and only if $a \in V(b)$ for $a \in \mathcal{A}$ and $b \in \mathcal{B}$. Then, by \upd{the assumption that $N_u$ can be oriented into $N$}, there is a one-to-one correspondence between the set of leaves of the tree $T_{N_u}$ and the set of pendant edges $\{\{t_1, \tilde{t_1}\}, \dots, \{t_n, \tilde{t_n}\}, \{s, \tilde{s}\}\}$ of $N_u$, where we used the argument in Remark \ref{rem:no.null}. For each cut vertex $v_i$ of $N_u$ in $V(G)$, let $a_i$ denote the corresponding vertex of $T_{N_u}$, let $b$ denote the vertex of $T_{N_u}$ corresponding to the block of $N_u$ containing $G$, and let $Z$ be the set of vertices of $T_{N_u}$ corresponding to the set $\{\tilde{t_1}, \dots, \tilde{t_n}, \tilde{s}\}$, that is cut vertices of $N_u$ incident to pendant edges. Then, for any $i \in [1, k]$, in the component $C(a_i)$ of $T_{N_u} - \{a_i, b\}$ containing $a_i$, there exists a path $\pi_i$ from $a_i$ to some $z_i \in Z$.
Let $P_i$ be the path in $N^x_u$ from $v_i$ to $z_i$ corresponding to each $\pi_i$, with the edge $\{z_i, t\}$ appended. 
It follows that each $P_i$ is a path in $N^x_u$ from $v_i$ to $t$, and the fact that each $\pi_i$ does not include $b$ ensures that each $P_i$ does not include any edge of $B$, so $V(P_i) \cap V(B) = \{v_i\}$. Moreover, since $C(a_i)$ and $C(a_j)$ are disjoint for $i \neq j$, $V(P_i) \cap V(P_j) = \{t\}$.
Hence, $\{P_1, \dots, P_k\}$ satisfies conditions 1--3.
\end{proof}

\begin{figure}[htbp]
\centering
\includegraphics[width=.7\textwidth]{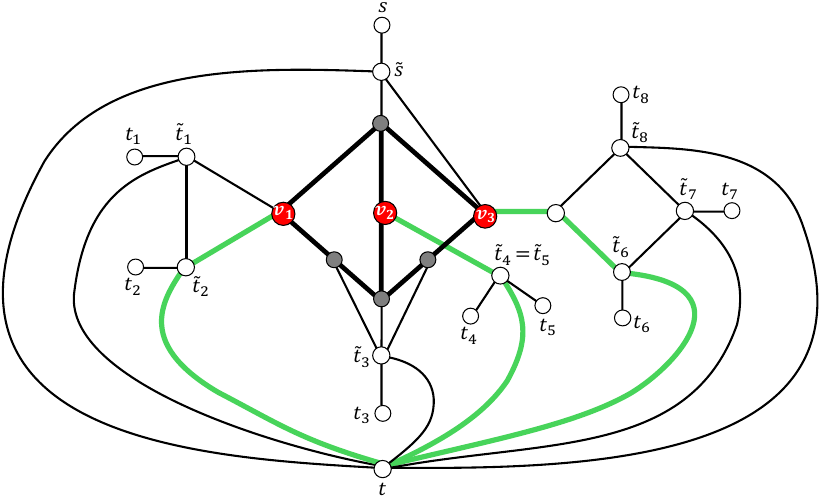}
\caption{The line segments represent $E(N_u)$, the curves represent $E(N^x_u) \setminus E(N_u)$. The filled vertices denote $V(G)$, in particular, $V_1 = \{v_1, v_2, v_3\}$ is shown in red. The thick black segments represent $E(G)$. In this situation, the collection of three green paths $[v_1, t], [v_2, t], [v_3, t]$ is one example that satisfies the conditions of Lemma~\ref{lem:disjoint_path}.
\label{fig:disjoint_path_N_ast}
}
\end{figure}

\begin{proposition}\label{prop:part1}
Let $N$ be a directed phylogenetic network. If $N$ is terminal planar, then for any $i \in [1,6]$, the cut-labeled graph $L(N_u)$ of $N$ does not contain an $\mathcal{H}_i$ structure.
\end{proposition}

\begin{proof}
Let $N$ be a terminal planar directed phylogenetic network with source $s$ and sink set $\{t_1, \dots, t_n\}$.
Suppose for contradiction, that $L(N_u) = (N_u, \ell)$ contains some minimal $\mathcal{H}_i$ structure $(F, \ell|_F)$ for some $i \in [1,6]$. It suffices to only consider $i \in \{2, 3, 5, 6\}$.
This is because $N_u$ is planar, by Theorem~\ref{thm:kuratowski}, $N_u$ does not contain any subgraph that is a subdivision of $K_5$ or $K_{3,3}$. However, if $i = 1$, then $F \subseteq N_u$ is a subdivision of $K_{3,3}$, and if $i = 4$, then $F \subseteq N_u$ is a subdivision of $K_5$.
Therefore, it is sufficient to derive a contradiction for each of the following four cases:
i) $i = 2$, ii) $i = 3$, iii) $i = 5$, and iv) $i = 6$.
Recall from \upd{Theorem~\ref{lem:N_completion}} that $N$ is terminal planar if and only if the underlying graph $N^x_u$ of the terminal cut-completion $N^x$ of $N$ is planar.
We shall show that, in each case i)--iv), $N^x_u$ contains a subdivision of $K_5$ or $K_{3,3}$.

\begin{enumerate}[i)]
\item In the case $i=2$,
$(F, \ell|_F)$ is label-preserving homeomorphic to some $\mathcal{H}_2$ graph $(H, h)$, and by Proposition~\ref{prop.map.phi}, $F$ is isomorphic to a subdivision of $H$.
Since $F \subseteq N_u \subseteq N^x_u$, $N^x_u$ also contains $F$.
$H$ has three vertices $v_1, v_2, v_3$ with $\deg_H(v_j) \in \{1, 2\}$.
Recall that $H$ is obtained by label-preserving contraction of zero or more pendant edges of $(H_2, h_2)$ in Figure~\ref{fig:forbidden},
and thus $h(v_j) = 1$ for $j \in [1,3]$.
By Proposition~\ref{prop:F}, each $\psi(v_j) \in V(F)$ is a cut vertex of $N_u$. Also by Proposition~\ref{prop:F}, $F$ is contained in a block of $N_u$. Therefore, regarding $\{\psi(v_1), \psi(v_2), \psi(v_3)\}$ as $V_1$ in Lemma~\ref{lem:disjoint_path},
$N^x_u$ contains a set $\{P_1, P_2, P_3\}$ of undirected paths satisfying the conditions 1--3 of Lemma~\ref{lem:disjoint_path}.
Since $F, P_1, P_2, P_3 \subseteq N^x_u$, $F \cup P_1 \cup P_2 \cup P_3$ is a subgraph of $N^x_u$. By conditions 1 and 3 of Lemma~\ref{lem:disjoint_path}, each $P_j$ is a path connecting $\psi(v_j) \in V_1$ to the sink $t \in V(N^x_u)$ of $N^x$, with $V(P_j) \cap V(F) = \{\psi(v_j)\}$.
By condition 2, $V(P_j) \cap V(P_k) = \{t\}$ for $j \neq k$.
Therefore, as shown in Figure~\ref{fig:F_K33_K5}(i), $F \cup P_1 \cup P_2 \cup P_3$ is a subdivision of $K_{3,3}$.

\item In the case $i = 3$,
$(F, \ell|_F)$ is label-preserving homeomorphic to some $\mathcal{H}_3$ graph $(H, h)$, and by Proposition~\ref{prop.map.phi}, $F$ is isomorphic to a subdivision of $H$.
By the same reasoning as in case i), $N^x_u$ also contains $F$.
$H$ has two vertices $v_1, v_2$ with $\deg_H(v_j) \in \{1, 2\}$.
Recall that $H$ is obtained by label-preserving contraction of zero or more pendant edges of $(H_3, h_3)$ in Figure~\ref{fig:forbidden},
and thus $h(v_j) = 1$ for $j \in [1,2]$.
By Proposition~\ref{prop:F}, each $\psi(v_j) \in V(F)$ is a cut vertex of $N_u$. Also by Proposition~\ref{prop:F}, $F$ is contained in a block of $N_u$. Therefore, regarding $\{\psi(v_1), \psi(v_2)\}$ as $V_1$ in Lemma~\ref{lem:disjoint_path},
$N^x_u$ contains a set $\{P_1, P_2\}$ of undirected paths satisfying the conditions 1--3 of Lemma~\ref{lem:disjoint_path}.
Since $F, P_1, P_2 \subseteq N^x_u$, $F \cup P_1 \cup P_2$ is a subgraph of $N^x_u$.
By the same reasoning as in case i), as shown in Figure~\ref{fig:F_K33_K5}(ii), $F \cup P_1 \cup P_2$ is a subdivision of $K_{3,3}$.

\item In the case $i = 5$,
$(F, \ell|_F)$ is label-preserving homeomorphic to some $\mathcal{H}_5$ graph $(H, h)$, and by Proposition~\ref{prop.map.phi}, $F$ is isomorphic to a subdivision of $H$.
By the same reasoning as in case i), $N^x_u$ also contains $F$.
$H$ has four vertices $v_1, v_2, v_3, v_4$ with $\deg_H(v_j) \in \{1, 3\}$.
Recall that $H$ is obtained by label-preserving contraction of zero or more pendant edges of $(H_5, h_5)$ in Figure~\ref{fig:forbidden},
and thus $h(v_j) = 1$ for $j \in [1,4]$.
By Proposition~\ref{prop:F}, each $\psi(v_j) \in V(F)$ is a cut vertex of $N_u$. Also by Proposition~\ref{prop:F}, $F$ is contained in a block of $N_u$. Therefore, regarding $\{\psi(v_1), \psi(v_2), \psi(v_3), \psi(v_4)\}$ as $V_1$ in Lemma~\ref{lem:disjoint_path},
$N^x_u$ contains a set $\{P_1, P_2, P_3, P_4\}$ of undirected paths satisfying the conditions 1--3 of Lemma~\ref{lem:disjoint_path}.
Since $F, P_1, P_2, P_3, P_4 \subseteq N^x_u$, $F \cup P_1 \cup P_2 \cup P_3 \cup P_4$ is a subgraph of $N^x_u$.
By the same reasoning as in case i), as shown in Figure~\ref{fig:F_K33_K5}(iii), $F \cup P_1 \cup P_2 \cup P_3 \cup P_4$ is a subdivision of $K_5$.

\item In the case $i = 6$,
$(F, \ell|_F)$ is label-preserving homeomorphic to some $\mathcal{H}_6$ graph $(H, h)$, and by Proposition~\ref{prop.map.phi}, $F$ is isomorrphic to a subdivision of $H$.
By the same reasoning as in case i), $N^x_u$ also contains $F$.
$H$ has two vertices $v_1, v_2$ with $\deg_H(v_j) \in \{1, 3\}$.
Recall that $H$ is obtained by label-preserving contraction of zero or more pendant edges of $(H_6, h_6)$ in Figure~\ref{fig:forbidden},
and thus $h(v_j) = 1$ for $j \in [1,2]$.
By Proposition~\ref{prop:F}, each $\psi(v_j) \in V(F)$ is a cut vertex of $N_u$. Also by Proposition~\ref{prop:F}, $F$ is contained in a block of $N_u$. Therefore, regarding $\{\psi(v_1), \psi(v_2)\}$ as $V_1$ in Lemma~\ref{lem:disjoint_path},
$N^x_u$ contains a set $\{P_1, P_2\}$ of undirected paths satisfying the conditions 1--3 of Lemma~\ref{lem:disjoint_path}.
Since $F, P_1, P_2 \subseteq N^x_u$, $F \cup P_1 \cup P_2$ is a subgraph of $N^x_u$.
By the same reasoning as in case i), as shown in Figure~\ref{fig:F_K33_K5}(iv), $F \cup P_1 \cup P_2$ is a subdivision of $K_5$.
\end{enumerate}

Therefore, in each of the cases i)--iv), by Theorem~\ref{thm:kuratowski}, $N^x_u$ is not planar. However, by \upd{Theorem~\ref{lem:N_completion}}, if $N$ is terminal planar, then $N^x$ is planar,  and hence $N^x_u$ is also planar. This is a contradiction. Thus, if $N$ is terminal planar, then $L(N_u)$ does not contain any of the structures $\mathcal{H}_1, \dots, \mathcal{H}_6$.
\end{proof}

\begin{figure}[htbp]
\centering
\includegraphics[width=.9\textwidth]{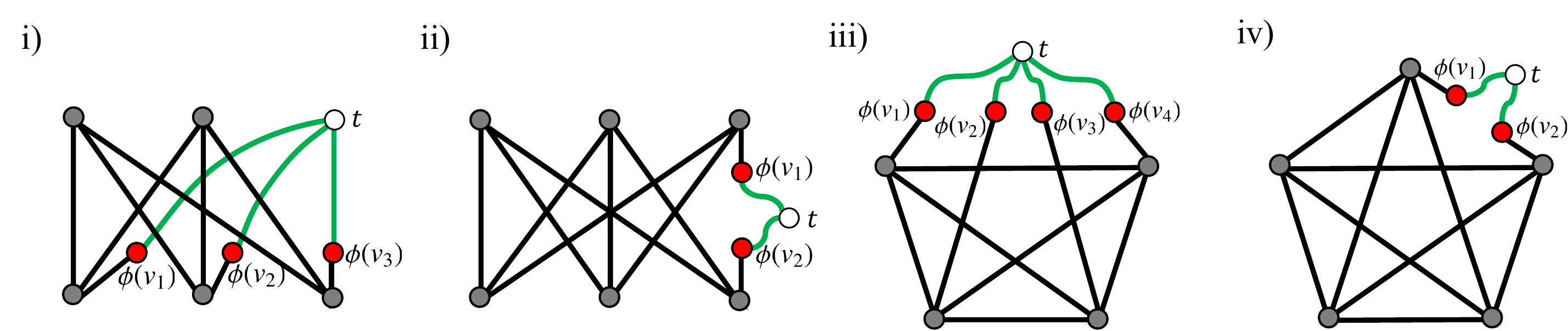}
\caption{An example of a subgraph $(F, \ell|_F)$ of $N^x_u$ discussed in the proof of Propositionon~\ref{prop:part1} in Cases i)--iv) (note that only the cases where each $v_j$ satisfies $\deg_H(v_j)=1$ are illustrated here). In this figure, every line connecting two vertices represents a path in $N^x_u$, and $E(F)$ denotes the union of all edges contained in the black paths. The cut vertices $\psi(v_j)$ of $N_u$ contained in $V(F)$ are shown in red, and for each $\psi(v_j)$, a path $P_j$ satisfying Conditions 1--3 of Lemma~\ref{lem:disjoint_path} is depicted as a green line.
\label{fig:F_K33_K5}}
\end{figure}

\subsection{Forbidden Structures for Terminal Planar Networks: Proof of Sufficiency}

\begin{proposition}\label{prop:part2}
Let $N$ be a directed phylogenetic network. If $N$ is not terminal planar, then there exists $i \in [1,6]$ such that the cut-labeled graph $L(N_u)$ of $N$ contains an $\mathcal{H}_i$ structure.
\end{proposition}

\begin{proof}
If $N$ is not planar, then by Theorem~\ref{thm:kuratowski}, $N_u \approx K_{3,3}$ or $N_u \approx K_{5}$, so $L(N_u)$ contains either an $\mathcal{H}_1$ or $\mathcal{H}_4$ structure. 
Assume $N$ is planar but not terminal planar. We show that $L(N_u)$ contains an $\mathcal{H}_i$ structure for some $i \in \{2,3,5,6\}$.
By \upd{Theorem~\ref{lem:N_completion}}, the above assumption is equivalent to $N$ being planar while $N^x$ is not planar. By Proposition~\ref{prop:relation_completion}, if $N^x$ is not planar, then $N^c$ is also not planar. Thus, $N_u$ is planar, but the supergraph $N^c_u$ is not. Therefore, by Theorem~\ref{thm:kuratowski}, there exists a subgraph $G$ of $N^c_u$ that is homeomorphic to $K_5$ or $K_{3,3}$, with $G \not\subseteq N_u$.
Let $t$ denote the sink of $N^c$. By Definition~\ref{dfn:cut_completion}, $V(N^c_u) = V(N_u) \cup \{t\}$ and $E(N^c_u) = E(N_u) \cup \{ \{v_1, t\}, \dots, \{v_k, t\} \}$, so $G \not\subseteq N_u$ and $G \subseteq N^c_u$ implies $t \in V(G)$.
Since $G \approx K_{3,3}$ or $G \approx K_5$, and $\deg_G(t) = 2$ or $\deg_G(t) \neq 2$, we consider the following four cases.
Let $L(N_u) = (N_u, \ell)$.

\begin{itemize}
  \item \textbf{Case 1:} $G \approx K_{3,3}$ and $\deg_G(t) \neq 2$.
  Since $G \approx K_{3,3}$ and $\deg_G(t) = 3$, $G-t$ is homeomorphic to a graph obtained from $K_{3,3}^{(-v)}(\simeq H_2)$ by contracting zero or more pendant edges.
  By Definition~\ref{dfn:cut_completion}, any subgraph of $N^c_u$ that excludes $t$ and its incident edges is a subgraph of $N_u$, so $G-t \subseteq N_u$.
  Furthermore, $G-t$ contains three distinct cut vertices $v_1, v_2, v_3$ of $N_u$, and by the definition of $L(N_u)$, we have $\ell(v_i) = 1$ for all $i \in [1,3]$.
  Let $\{t, w_1, w_2\} \sqcup \{u_1, u_2, u_3\} = \{ v \in V(G) \mid \deg_G(v) = 3 \}$. Then $G$ is the union of a graph $G'$ homeomorphic to $K_{2,3}$ and three internally vertex-disjoint $u_i$-$t$ paths $[u_1, t]_G, [u_2, t]_G, [u_3, t]_G$, where each $v_i$ satisfies $[v_i, t]_G \subseteq [u_i, t]_G$.
   Since $K_{2,3}$ is 2-connected, $G'$ contains no cut edge of $N_u$, thus $\ell(e) = 0$ for every edge $e$ of $G'$. If $u_i = v_i$ for all $i$, then $(G', \ell|_{G'})$ is an $\mathcal{H}_2$ structure in $(N_u, \ell)$.
  Next, if $u_1 \neq v_1$, since $G$ contains no degree-1 vertex, $[u_1, v_1]_G$ contains no pendant edge of $N_u$. 
  Without loss of generality, we may assume that $[u_1, v_1]_G$ contains no cut edge of $N_u$, by choosing $v_1$ to be the cut vertex of $N_u$ closest to $u_1$ along the path $[u_1, t]_G$.
  By redefining $G' \cup [u_1, v_1]_G$ as $G'$, $(G', \ell|_{G'})$ is an $\mathcal{H}_2$ structure.This reasoning applies for any number of indices $i$ such that $u_i \neq v_i$.
  Therefore, $(N_u, \ell)$ contains an $\mathcal{H}_2$ structure.

  \item \textbf{Case 2:} $G \approx K_{3,3}$ and $\deg_G(t) = 2$.
  Since $G \approx K_{3,3}$ and $\deg_G(t) = 2$, $G-t$ is homeomorphic to a graph obtained from $K_{3,3}^{(-e)} (\simeq H_3)$ by contracting zero or more pendant edges.
  As in Case 1, $G-t \subseteq N_u$, and $G-t$ contains two distinct cut vertices $v_1, v_2$ of $N_u$, with $\ell(v_i) = 1$ for $i \in [1,2]$.
  Let $\{u_1, w_1, w_2\} \sqcup \{u_2, w_3, w_4\} = \{ v \in V(G) \mid \deg_G(v) = 3 \}$. Then $G$ is the union of a graph $G'$ homeomorphic to $K_{3,3}^{-}$ and two internally vertex-disjoint $u_i$-$t$ paths $[u_1, t]_G, [u_2, t]_G$, with each $v_i$ satisfying $[v_i, t]_G \subseteq [u_i, t]_G$.
  Since $K_{3,3}^{-}$ is 2-connected, $G'$ contains no cut edge of $N_u$, thus $\ell(e) = 0$ for any edge $e$ of $G'$. If $u_i = v_i$ for both $i$, then $(G', \ell|_{G'})$ is an $\mathcal{H}_3$ structure. If $u_1 \neq v_1$, using the same reasoning as in Case 1, $[u_1, v_1]_G$ can be assumed to contain no cut edge. 
  By redefining $G' \cup [u_1, v_1]_G$ as $G'$, $(G', \ell|_{G'})$ is an $\mathcal{H}_3$ structure. This applies for any number of $u_i \neq v_i$. 
  Therefore, $(N_u, \ell)$ contains an $\mathcal{H}_3$ structure.

  \item \textbf{Case 3:} $G \approx K_5$ and $\deg_G(t) \neq 2$.
  For $G \approx K_5$ and $\deg_G(t) = 4$, $G-t$ is homeomorphic to a graph obtained from $K_5^{(-v)} (\simeq H_5)$ by contracting zero or more pendant edges.
  By similar logic, $G-t \subseteq N_u$ and $G-t$ contains four distinct cut vertices $v_1, v_2, v_3, v_4$ of $N_u$ with $\ell(v_i) = 1$ for all $i \in [1,4]$.
  Let $\{t, u_1, u_2, u_3, u_4\} = \{ v \in V(G) \mid \deg_G(v) = 4 \}$. Then $G$ is the union of a graph $G'$ homeomorphic to $K_4$ and four internally vertex-disjoint $u_i$-$t$ paths $[u_1, t]_G, [u_2, t]_G, [u_3, t]_G, [u_4, t]_G$ with each $v_i$ satisfying $[v_i, t]_G \subseteq [u_i, t]_G$.
  Since $K_4$ is 2-connected, $G'$ contains no cut edge of $N_u$, thus $\ell(e) = 0$ for every edge $e$ of $G'$. If $u_i = v_i$ for all $i$, $(G', \ell|_{G'})$ is an $\mathcal{H}_5$ structure. If $u_1 \neq v_1$, by the same reasoning as in Case 1, $[u_1, v_1]_G$ can be chosen to contain no cut edge. By redefining $G' \cup [u_1, v_1]_G$ as $G'$, $(G', \ell|_{G'})$ is an $\mathcal{H}_5$ structure. This applies for any number of $u_i \neq v_i$. 
Therefore, $(N_u, \ell)$ contains an $\mathcal{H}_5$ structure.

  \item \textbf{Case 4:} $G \approx K_5$ and $\deg_G(t) = 2$.
  For $G \approx K_5$ and $\deg_G(t) = 2$, $G-t$ is homeomorphic to a graph obtained from $K_5^{(-e)} (\simeq H_6)$ by contracting zero or more pendant edges.
  Similarly, $G-t \subseteq N_u$, and $G-t$ contains two distinct cut vertices $v_1, v_2$ of $N_u$ with $\ell(v_i) = 1$ for $i \in [1,2]$.
  Let $\{u_1, u_2, w_1, w_2, w_3\} = \{ v \in V(G) \mid \deg_G(v) = 4 \}$. Then $G$ is the union of a graph $G'$ homeomorphic to $K_4^{-}$ and two internally vertex-disjoint $u_i$-$t$ paths $[u_1, t]_G, [u_2, t]_G$, with each $v_i$ satisfying $[v_i, t]_G \subseteq [u_i, t]_G$.
   Since $K_4^{-}$ is 2-connected, $G'$ contains no cut edge of $N_u$, thus $\ell(e) = 0$ for any edge $e$ of $G'$. If $u_i = v_i$, $(G', \ell|_{G'})$ is an $\mathcal{H}_6$ structure. If $u_1 \neq v_1$,by the same reasoning as in Case 1, $[u_1, v_1]_G$ can be assumed to contain no cut edge. By redefining $G' \cup [u_1, v_1]_G$ as $G'$, $(G', \ell|_{G'})$ is an $\mathcal{H}_6$ structure. This applies for any number of $u_i \neq v_i$. 
  Therefore, $(N_u, \ell)$ contains an $\mathcal{H}_6$ structure.
\end{itemize}

Therefore, if $N$ is not terminal planar, then in any of Cases 1--4, $L(N_u)$ contains some $\mathcal{H}_i$ structure ($i \in [1,6]$).
\end{proof}

\subsection{Main Results}\label{sec:main}
Combining Propositions \ref{prop:part1} and \ref{prop:part2}, we obtain the following result, which is illustrated in Figure \ref{fig:main.theorem.demo}.
\begin{theorem}\label{thm:main}
Let $N$ be a directed phylogenetic network. Then, $N$ is terminal planar if and only if, for any $i \in [1,6]$, the cut-labeled graph $L(N_u)$ of $N$ contains no $\mathcal{H}_i$ structure.
\end{theorem}

\begin{figure}[htbp]
\centering
\includegraphics[width=0.65\textwidth]{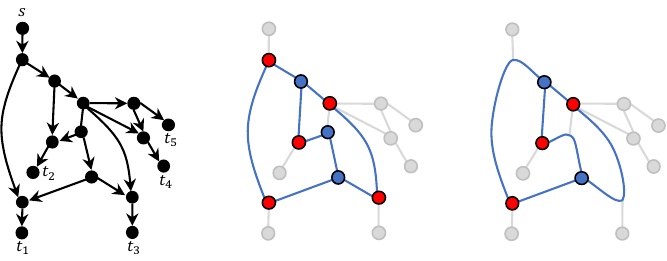}
\caption{An example of non-terminal planar networks. In this case,  the network contains an $\mathcal{H}_2$ structure. 
\label{fig:main.theorem.demo}
}
\end{figure}

Note that although Theorem~\ref{thm:main} concerns a directed phylogenetic network $N$, both Proposition~\ref{prop:part1} and Proposition~\ref{prop:part2} focus only on the cut-labeled graph of the underlying graph $N_u$ of $N$. Therefore, if an undirected phylogenetic network $G$ is isomorphic to the underlying graph $N_u$ of a directed phylogenetic network $N$, then Theorem~\ref{thm:main} applies equally to the cut-labeled graph of $G$ as well. This yields Corollary~\ref{cor:main_underlying}.

\begin{corollary}\label{cor:main_underlying}
Let $G$ be an undirected graph that is isomorphic to the underlying graph of a directed phylogenetic network. Then, $G$ is terminal planar if and only if, for any $i \in [1,6]$, the cut-labeled graph $L(G)$ contains no $\mathcal{H}_i$ structure.
\end{corollary}

An undirected phylogenetic network $G$ is called \emph{binary} if every non-leaf vertex $v$ of $G$ satisfies $\deg_G(v) \in \{2,3\}$. Similarly, a directed phylogenetic network $N$ is called \emph{binary} if every vertex $v$ of $N$ that is neither a root nor a leaf satisfies $\left(\mathrm{indeg}_N(v), \mathrm{outdeg}_N(v)\right) \in \{ (1,1), (2,1), (1,2) \}$.

By definition, the degree of each vertex in a binary phylogenetic network is at most $3$. On the other hand, for any $i \in [4,6]$, any graph that contains an $\mathcal{H}_i$ structure must have a vertex of degree at least $4$, so a cut-labeled graph of a binary phylogenetic network cannot contain an $\mathcal{H}_4$, $\mathcal{H}_5$, or $\mathcal{H}_6$ structure. It then follows from Theorem~\ref{thm:main} that we obtain Corollary~\ref{cor:main_binary} and Corollary~\ref{cor:main_binary_underlying}.

\begin{corollary}\label{cor:main_binary}
Let $N$ be a binary directed phylogenetic network. Then, $N$ is terminal planar if and only if the cut-labeled graph $L(N_u)$ of $N$ contains no $\mathcal{H}_1$, $\mathcal{H}_2$, or $\mathcal{H}_3$ structure.
\end{corollary}

\begin{corollary}\label{cor:main_binary_underlying}
Let $G$ be an undirected graph that is isomorphic to the underlying graph of a binary directed phylogenetic network. Then, $G$ is terminal planar if and only if the cut-labeled graph $L(G)$ contains no  $\mathcal{H}_1$, $\mathcal{H}_2$, or $\mathcal{H}_3$ structure.
\end{corollary}

Although we have obtained the forbidden subgraphs for terminal planar networks, we can describe more efficient, linear-time algorithms for terminal planarity testing and drawing without checking the existence of these obstructions. 
Given a directed phylogenetic network $N$---or its underlying undirected graph $N_u$---one can test its terminal planarity in linear time simply by checking the planarity of its $st$-completion $N^\ast$ by 
Theorem~\ref{lem:N_completion} (for linear-time planarity testing, see \cite{BOOTH1976335,Boyer_Myrvold_2004, HopcroftTarjan1974Planarity}). Moreover, a terminal planar drawing of $N$ can be obtained in linear time by computing a planar embedding of this $st$-digraph $N^\ast$ (using e.g.\ \cite{DIBATTISTA1988175,DGR1989Area}) and then removing the unique sink $t$ and its incident edges from the embedding. Our algorithms provide a distinct approach from the linear-time algorithms in \cite{WuMoulton9804871} that rely on testing  the upward planarity after $t$-completion $N^+$. Upward planarity can be determined in linear time for  single-source digraphs (\cite{bbmt-oupsst-1998} and~\cite[Theorem~15]{garg1995upward}) while being NP-complete for multi-source digraphs \cite{GargTamassia2001UpwardNP}.

\section{Application}\label{sec:application}
So far, we have obtained the forbidden subgraphs of terminal planar phylogenetic networks in both directed and undirected settings (Theorem \ref{thm:main} and Corollary \ref{cor:main_underlying}). In this section, we will show a quick application of Corollary \ref{cor:main_underlying}.  To be more specific, we will \upd{present} Theorem \ref{thm:general} \upd{which gives} a characterization of undirected graphs admitting a planar drawing in which all specified vertices lie on the outer face (e.g.\ Figure \ref{fig:ex_outerdrawable_graph}).
The key idea behind this result is the operation of transforming undirected planar graphs into  undirected planar phylogenetic networks, as is illustrated in Figure \ref{fig:ex_outerdrawable_graph}. 

\begin{figure}[htbp]
\centering
\includegraphics[width=1\textwidth]{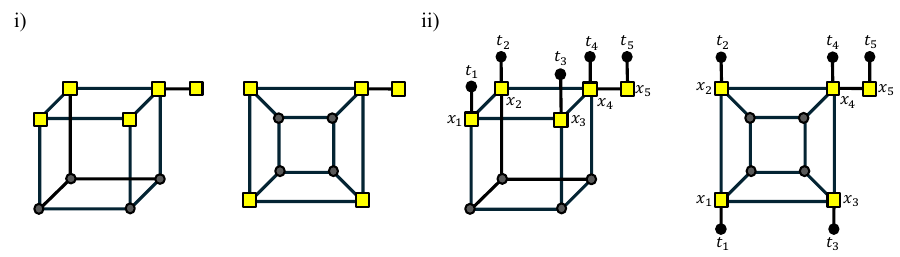}
\caption{i) An undirected graph $G$ with $V_o \subseteq V(G)$ such that there exists a planar drawing of $G$ where all vertices in the prescribed subset $V_o$ (highlighted in yellow) lie on the outer face. 
ii) An undirected phylogenetic network $\tilde{G}$ constructed from $G$ by adding a pendant edge $(x, t)$ for each $x\in V_o$, and a terminal planar drawing of $\tilde{G}$.  
\label{fig:ex_outerdrawable_graph}
}
\end{figure}

From now on, let $G$ be a planar connected undirected  graph, and let $V_o \subseteq V(G)$ be the set of vertices that we wish to place on the outer face in a planar drawing of $G$. Without loss of generality, we may assume that for every cut vertex $v$ of $G$, each connected component of $G - v$ contains at least one element of $V_o$. Indeed, when a connected component $C$ of $G - v$ has no element of $V_o$, we may choose any planar drawing of $C$ in finding a desirable planar drawing of $G$. 

In the proof of Theorem~\ref{thm:general}, we will use Theorem~\ref{thm:st-orientation}, which is a  restatement of classic results in \cite{lempel1967algorithm} (Theorem~a.1 and Lemma~a.1 in \cite{lempel1967algorithm}; see also Theorem 4.1 in \cite{DEFRAYSSEIX1995157}). Briefly, Lempel et al.\ \cite{lempel1967algorithm} showed that there exists such an acyclic orientation of $G$ that is described in Theorem \ref{thm:st-orientation} if and only if there exists a function called an $st$-numbering of $G$. For details, see \cite{lempel1967algorithm, DEFRAYSSEIX1995157}.

\begin{theorem}\label{thm:st-orientation}
Let $G$ be a connected undirected graph, and let $\{s,t\}$ be an edge of $G$. Then $G$ can be oriented to an acyclic digraph with a unique source $s$ and a unique sink $t$ if and only if $G$ is biconnected.
\end{theorem}

\begin{theorem}\label{thm:general}
Let $G$ be a planar connected undirected graph and let $V_o \subseteq V(G)$ such that for every cut vertex $v$ of $G$, each connected component of $G - v$ contains at least one vertex in $V_o$. Let $\omega: V(G) \cup E(G) \rightarrow \{0, 1\}$ be the labeling function defined by equation \eqref{eq:out}:
\begin{equation}\label{eq:out}
\omega(x) := 
\begin{cases}
1 & \text{if $x$ is a cut vertex or cut edge of $G$, or $x \in V_o$,} \\
0 & \text{otherwise.}
\end{cases}
\end{equation}
Then, $G$ admits a planar drawing in which all vertices in $V_o$ lie on the outer face if and only if the $0/1$-labeled graph $(G, \omega)$ contains no $\mathcal{H}_i$ structure for any $i \in [1,6]$.
\end{theorem}

\begin{proof}
For each vertex $x_i\in V_o$ of $G$, let $\tilde{G}$ be the graph obtained by adding a new vertex $t_i$ with $\deg_{\tilde{G}}(t_i)=1$ and an edge $\{x_i, t_i\}$. Let $T$ denote the set of added vertices $\{t_i\}$. By construction, $\tilde{G}$ is an undirected phylogenetic network with terminal set $T$. Moreover, $G$ admits a planar drawing in which all elements of $V_o$ lie on the outer face if and only if $\tilde{G}$ is terminal planar.

To apply Corollary~\ref{cor:main_underlying} to $\tilde{G}$, we show that $\tilde{G}$ is isomorphic to the underlying graph of some directed phylogenetic network. Let $s$ be an arbitrary vertex in $T$. Let $\hat{G}$ be the graph obtained from $\tilde{G}$ by adding a new vertex $t$ and edges $\{s_i, t\}$ for each $s_i\in T\setminus \{s\}$. Although $\hat{G}$ need not be biconnected, the graph $G^\star$ obtained by adding the edge $\{s, t\}$ to $\hat{G}$ is biconnected. By Theorem~\ref{thm:st-orientation}, there exists an acyclic orientation of $G^\star$ with $s$ as the unique source and $t$ as the unique sink.
Under such an orientation, the subgraph $\tilde{G}$ of $G^\star$ is also acyclic. Clearly, $\tilde{G}$ has a unique root $s$ and leaves connected to $t$.
Therefore, $\tilde{G}$ is the underlying graph of some directed phylogenetic network. Hence, by Corollary~\ref{cor:main_underlying}, $\tilde{G}$ is terminal planar if and only if $L(\tilde{G})$ contains no $\mathcal{H}_i$ structure for any $i\in[1,6]$.

It remains to show that the cut-labeled graph $L(\tilde{G})$ contains no $\mathcal{H}_i$ structure if and only if $(G, \omega)$ contains no $\mathcal{H}_i$ structure.
We claim that $(G, \omega)\subseteq L(\tilde{G})$. Indeed, by construction of $\tilde{G}$, each $x\in V_o$ is a cut vertex of $\tilde{G}$, so the label $\omega (x)$ of each element $x\in V(G)\cup E(G)$ is unchanged in $L(\tilde{G})$. In addition, each pendant edge $\{x_i,t_i\}$ in $L(\tilde{G})$ has label $1$ since it is a cut edge of $\tilde{G}$; however, for any $i\in[1,6]$, no $\mathcal{H}_i$ structure contains an edge with label $1$, so $L(\tilde{G})$ contains an $\mathcal{H}_i$ structure if and only if $(G,\omega)$ contains an $\mathcal{H}_i$ structure.
\end{proof}

\section{Conclusion and Open Problem}\label{sec:conclusion_open_problems}
\upd{We obtained two characterizations of terminal planar networks. The first one, which concerns directed networks, is based on the planarity of certain supergraphs (Theorem~\ref{lem:N_completion}). This characterization furnishes new linear-time algorithms for terminal planarity testing and drawing as described in Section~\ref{sec:main}. It also plays an important role in proving our main result, a Kuratowski-type characterization that is valid for both directed and undirected networks (Theorem \ref{thm:main} and Corollary \ref{cor:main_underlying}).} 
Building on the concept of cut-labeled graphs, we identified a finite family of forbidden subgraphs, denoted by the $\mathcal{H}_i$ structures, which are defined using the six graphs shown in Figure \ref{fig:forbidden}. 
As discussed in Section \ref{sec:application}, our main result readily applies to the more general problem of determining whether an undirected graph $G$ admits a planar drawing in which a specified set of vertices $V_o \subseteq V(G)$ all lie on the outer face.

It remains open whether one can obtain a Kuratowski-type theorem for upward planar networks, even in the single-source case.  It would be therefore interesting to try to reveal their forbidden subgraphs by weakening some of the $\mathcal{H}_i$ structures ($i\in[1,6]$) discussed in this paper. 
Since $\mathcal{H}_1$ and $\mathcal{H}_4$ ($K_{3,3}$ and $K_5$, respectively) are the forbidden structures of planar graphs, they are clearly forbidden subgraphs of upward planar networks. In contrast, it appears that $\mathcal{H}_2$ and $\mathcal{H}_5$ are not forbidden subgraphs of upward planar networks, as shown by the examples in Figure~\ref{fig:ex_H2_H5_upward}. We note that $\mathcal{H}_3$ and $\mathcal{H}_6$ may require more careful treatment; among single-source acyclic digraphs containing $\mathcal{H}_3$ or $\mathcal{H}_6$, some are upward planar while others are not, as illustrated in  Figure~\ref{fig:ex_H3_H6_notupward}.  

\begin{figure}[htbp]
\centering
\includegraphics[width=0.45\textwidth]{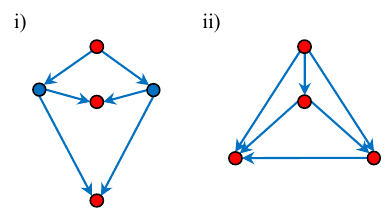}
\caption{Single-source upward planar networks that contain i) an $\mathcal{H}_2$ structure and ii) an $\mathcal{H}_5$ structure. 
\label{fig:ex_H2_H5_upward}
}
\end{figure}

\begin{figure}[htbp]
\centering
\includegraphics[width=0.9\textwidth]{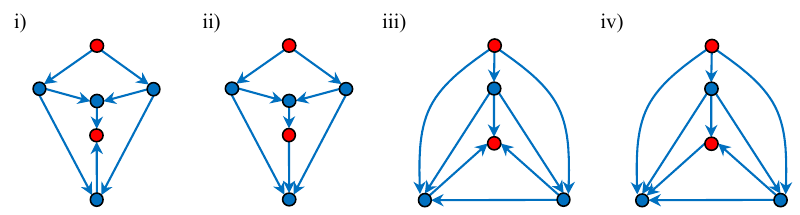}
\caption{i) A non-upward planar network containing an $\mathcal{H}_3$ structure, ii) an upward planar network containing an $\mathcal{H}_3$ structure, iii) a non-upward planar network containing an $\mathcal{H}_6$ structure, iv) an upward planar network containing an $\mathcal{H}_6$ structure. 
\label{fig:ex_H3_H6_notupward}
}
\end{figure}


\bibliographystyle{unsrt} 
\bibliography{planarity-manuscript-english.bib} 

\end{document}